\documentclass[a4paper,12pt]{amsart}
\usepackage[cp1251]{inputenc}
\usepackage[english, ukrainian]{babel}
\usepackage[all,dvipdfm]{xy}
\usepackage{amssymb,upref}
\usepackage[dvips]{graphicx}
\graphicspath{{C:\Users\Aspirant\Desktop\Новая папка}}
\usepackage{setspace}
\usepackage{xcolor}
\usepackage{fancyhdr}
\usepackage{graphicx}

\theoremstyle{plain}
\newtheorem{theorem}{Theorem}
\newtheorem{lemma}{Lemma}
\newtheorem{corollary}{Corollary}
\newtheorem{property}{Property}
\newtheorem{prop}{Proposition}
\newtheorem{proposition}{Proposition}
\newtheorem{remark}{Remark}

\newtheorem{statment}{Statement}
\newtheorem*{main_theorem}{Main Theorem}
\theoremstyle{definition}
\newtheorem{definition}{Definition}
\theoremstyle{remark}

\uccode`ґ=`Ґ \uccode`Ґ=`Ґ \lccode`Ґ=`ґ \lccode`ґ=`ґ %
\uccode`є=`Є \uccode`Є=`Є \lccode`Є=`є \lccode`є=`є %
\uccode`і=`І \uccode`І=`І \lccode`І=`і \lccode`і=`і %
\uccode`ї=`Ї \uccode`Ї=`Ї \lccode`Ї=`ї \lccode`ї=`ї %

\begin{document}
\selectlanguage{english}

\title[] {Structure and minimal generating sets of Sylow 2-subgroups of alternating groups, properties of its commutator subgroup} 
\author[]{Ruslan Skuratovskii}


\begin{abstract}

In this article the research of Sylows $p$-subgroups of  ${{A}_{n}}$ and ${{S}_{n}}$, which was started in \cite{Dm,Sk, Paw} is continued.
Let $Syl_2{A_{2^k}}$ and $Syl_2{A_{n}}$ be Sylow 2-subgroups of corresponding alternating groups $A_{2^k}$ and $A_{n}$. 
We find a least generating set and a structure for such subgroups $Sy{{l}_{2}}{{A}_{{{2}^{k}}}}$ and $Syl_2{A_{n}}$ and commutator width of $Syl_2{A_{2^k}}$ \cite{Mur}.

	The purpose of this paper is to research the structure of a Sylow 2-subgroups and to construct a minimal generating set for such subgroups.
 The main result is to prove minimality of this generating set for the above indicated subgroups and also the description of their structure.


\end{abstract}

\maketitle

\section{Introduction }
The aim of this paper is to research the structure of Sylow 2-subgroups of $A_{2^k}$, $A_n$ and to construct a minimal generating set for $Syl_2 {A_{2^k}}$. The case of Sylow subgroup where $p=2$ is very special because group $C_2\wr C_2\wr C_2 \ldots \wr C_2 $ admits odd permutations, this case was not
fully investigated in \cite{Dm,Sk}. The authors of \cite{Dm, Paw} didn't prove minimality of found by them system of generators for such Sylow 2-subgroups of $A_n$ and structure of a Sylow 2-subgroups was found by them not fully, 
 for case $n=2^k$ structure of $Syl_2 A_n$ was not found.
 This question is up till now under consideration.
There was a mistake in a statement about irreducibility of a set of $k+1$ elements for $Syl_2(A_{2^k})$ that was in abstract \cite{Iv} on Four ukraine conference
of young scientists in 2015 year.
These groups have applications in the automaton theory, because if all states of automaton $A$ have output function that can be presented as cycle $(1,2,...,p)$ then group $G_A(X)$ of this automaton, where $X$ is finite alphabet, is Sylows $p$-subgroup of the group of all automaton transformations $GA(X)$ \cite{GrNe}. Also in this case holds $Syl_p(AutX)>FGA(X)$ \cite{GrNe}, where $ \mid X \mid \in \mathbb{N} $. Thus, finding the minimum size of generating set  number is important. Let $X^*$ be the free
 monoid freely generated by $X$.
All undeclared terms are from \cite{Ne, Gr}.

\section{ Preliminaries}
 The set $X^*$ is naturally a vertex set of a regular rooted tree, i.e. a connected graph without cycles
and a designated vertex $v_0$ called the root, in which two words are connected by an edge if and only if they are of form $v$ and $vx$, where $v\in X^*$, $x\in X$.
The set $X^n \subset X^*$ is called the $n$-th level of the tree $X^*$
and $X^0 = \{v_0\}$. We denote by $v_{j,i}$ the vertex of $X^j$, which has the number $i$.
The subtree of $X^{*}$ induced by the set of vertices $\cup_{i=0}^k X^i$ is denoted by $X^{[k]}$.
Note that the unique vertex $v_{k,i}$ corresponds to the unique word $v$ in alphabet $X$.
For every automorphism $g\in Aut{{X}^{*}}$ and every word $v \in X^{*}$  define the section (state) $g_{(v)} \in AutX^{*}$ of $g$ at $v$ by the rule: $g_{(v)}(x) = y$ for $x, y \in X^*$  if and only if $g(vx) = g(v)y$.
 The restriction of the action of an automorphism $g\in AutX^*$ to the subtree $X^{[l]}$ is denoted by $g_{(v)}|_{X^{[l]}}$.
 A restriction $g_{(v)}|_{X^{[1]}} $ is called the vertex permutation (v.p.) of $g$ in a vertex $v$. 
Let us introduce conventional signs for a v.p. state value of $\alpha $ in $v_{ki}$ as ${{s}_{ki}}(\alpha )$ we put that ${{s}_{ki}}(\alpha )=1$ if $\alpha_{(v_{ki})} |_{X^{[1]}}(x)=y,\,\,\,x\ne y $  
 such state of v.p. is active, and $s_{ki}(\alpha )=0$ if $\alpha_{(v_{ki})} |_{X^{[1]}}(x)=x $
 such state of v.p. is trivial.
Let us label every vertex of ${{X}^{l}},\,\,\,0\le l<k$ by sign 0 or 1 in relation to state of v.p. in it. Obtained by such way a vertex-labeled regular tree is element of $Aut{{X}^{[k]}}$.

We denote by $v_{j,i} X^{[k-j]} $  subtree of $X^{[k]}$ with a root in $v_{j,i}$.
An automorphism of $X^{[k]}$ with non-trivial states of v.p. in some of
 $v_{1,1}$, $v_{1,{2}}$, $v_{2,{1}}$,..., $v_{2,4}$, ... ,$v_{m,1}$, ... ,$v_{m,j}$, $ m < k, j \leq 2^m$ is denoted by $\mathop{\beta}_{1,(i_{11},i_{12});...; l,(i_{l1},...,i_{l2^l});...;{m},(i_{m1},...,i_{m2^m})}$
where the index that stands straight before parentheses are number of level in
parentheses we write a
tuple
  of states of v.p.  of this level. In other words we set $i_{mj}=0$ if v.p. in $v_{mj}$ is trivial, $i_{mj}=1$ in other case, i.e.  $i_{mj} = {{s}_{mj}}(\mathop{\beta})$, where $\mathop{\beta}\in AutX^{[k]}$, $m<k$.
  If for some $l$ all $i_{lj}=0$ then
     $2^l$-tuple $l,(i_{l1} ,..., i_{l 2^l})$ does not figure in indexes of $\beta$. But if numbers of active vertices are certain, for example $v_{j,1}$ and $v_{j,s}$, we can use more easy notation $\mathop{\beta}_{j,(1,s);}$, where in parentheses numbers of vertices with active state of v.p. from level $j$. If in parentheses only one index then parentheses can be omitted for instance $\mathop{\beta}_{j,(s);}=\mathop{\beta}_{j,s;}$.
Denote by $\tau_{i,...,j}$ the automorphism of $X^{[k]}$, which has a non-trivial v.p.
only in vertices $v_{k-1,i}$, ... ,$v_{k-1, j}$, $j  \leq 2^{k-1} $ of the
level $X^{k-1}$.
Denote by $\tau$ the automorphism $\tau_{1,2^{k-1}}$.
Let us consider special elements such that: $ \mathop{\alpha}_{0}=\mathop{\beta}_{0}=\mathop{\beta}_{0,(1,0,...,0)}, \mathop{\alpha}_{1}=\mathop{\beta}_{1}=\mathop{\beta}_{1,(1,0,...,0)}
, \ldots ,\mathop{\alpha}_{l}=\mathop{\beta}_{l}=\mathop{\beta}_{l,(1,0,...,0)} $.

\section{ Main result  }
Recall that a wreath product of permutation groups is associative construction.
We consider $C_2$ as additive group with two elements 0, 1. For constructing a wreath product we define an action of $C_2$ by shift on $X=\{0, 1\}$. As well known that $AutX^{[k-1]} \simeq \underbrace {C_2 \wr ...\wr C_2}_{k-1}$ \cite{GrNe}.

 \begin{lemma} \label{even} 

  Every automorphism that has active v.p. only on $X^l$, $l<k-1$ acts by even permutation on $X^{k}$.

\end{lemma}
\begin{proof}
Actually every transposition in vertex from $X^l$, $l<k-1$ acts on even number of pair of vertexes because of binary tree structure. More precisely it realize even permutation on the set $X^{k}$ with cyclic structure \cite{Sh}  $(1^{2^{k-1}-2^{k-l-l}},2^{2^{k-l-l}})$ because it formed by the structure of binary tree.
 \end{proof}
\begin{corollary}\label{B_k-1}
Due to Lemma \ref{even} automorphisms from $Aut X^{[k-1]}=\langle \mathop{\alpha}_{0}. ... , \mathop{\alpha}_{k-2} \rangle$ form a group ${B_{k-1}}= \underbrace {C_2 \wr ...\wr C_2}_{k-1} $ which acts on
$X^{k-1}$ by even permutations. Size of $B_{k-1}$ equal to $2^{2^{k-1}-1}$.
\end{corollary}

Let us denote by $W_{k-1}$ the subgroup of $Aut X^{[k]}$ such that has active states only on $ X^{k-1}$ and number of such states is even, i.e. $W_{k-1} \vartriangleleft  St_{G_k}(k-1)$ \cite{Ne}.

\begin{prop} \label{ordW}

The size of ${{W}_{k-1}}$ is equal to ${{2}^{{{2}^{k-1}}-1}},\,\,k > 1$ and its structure is $(C_2)^{{{2}^{k-1}}-1}$.
\end{prop}

  \begin{proof}
On ${{X}^{k-1}}$ we have ${{2}^{k-1}}$ vertices where can be elements of a group ${{V}_{k-1}}\simeq {{C}_{2}}\times {{C}_{2}}\times ...\times {{C}_{2}}\simeq {{({{C}_{2}})}^{k-1}}$, but as a result of the fact that ${X}^{k-1}$ contains only even number of non trivial v.p. from $X^{k-1}$, there are only half of all permutations from ${{V}_{k-1}} \simeq St_{G_k}(k-1)$ on $X^{k-1}$. So it is subgroup ${{W}_{k-1}}\simeq {}^{C_{2}^{{{2}^{k-1}}}}/{}_{{{C}_{2}}}$ of ${{V}_{k-1}}$. So we can state that $|{{W}_{k-1}}|=2^{k-1}-1$, $W_{k-1}$ has $k-1$ generators and we can consider ${W}_{k-1}$ as vector space of dimension $k-1$.
 \end{proof}

For example let us consider the subgroup $W_{4-1}$ of $A_{2^4}$ its size is $2^{2^{4-1}-1}=2^7$ and $|A_{2^4}|=2^{14}$.
Let us denote by $G_k$ the subgroup of $Aut X^{[k]}$ such that $G_k  \simeq  B_{k-1} \ltimes W_{k-1}$.
 \begin{lemma}\label{gen} 
 The elements $\tau $ and ${{\alpha }_{0}},...,{{\alpha }_{k-1}}$ generate arbitrary element ${{\tau }_{ij}}$.
\end{lemma}
 \begin{proof} According to \cite{Gr,Sk} 
  the set  ${{\alpha }_{0}},...,{{\alpha }_{k-2}}$ is minimal generating set for group $Aut{{X}^{[k-1]}}$. Firstly, we shall prove the possibility of generating an arbitrary $\tau_{ij}$, from set $v_{(k-1,i)}$, $1\leq i \leq 2^{k-2} $. 
Since $Aut v_{1,1}X^{[k-2]} \simeq \left\langle {{\alpha }_{1}},...,{{\alpha }_{k-2}} \right\rangle $ acts on ${{X}^{k-1}}$ transitively from it follows existing of an ability to permute vertex with a transposition from automorphism $\tau $ and stands in $v_{k-1,1}$ in arbitrary vertex ${{v}_{k-1,j}},\,\,\,j\le {{2}^{k-2}}$ of $v_{1,1}X^{[k-1]}$. For this goal we act by $\alpha_{k-j}$ at $\tau $: $\alpha_{k-j} \tau \alpha_{k-j} ={{\tau }_{j, 2^{k-2}}}$. Similarly we act on $\tau$ by corespondent $\alpha_{k-i}$ to get $\tau_{i, 2^{k-2}}$ from $\tau $: $\alpha_{k-i}  \tau \alpha_{k-i}^{-1}={{\tau }_{i, 2^{k-2}}}$. Note that automorphisms $\alpha_{k-j}$ and $\alpha_{k-i}, 1<i,j<k-1$ acts only on subtree $v_{1,1}X^{[k-1]}$ that's why they fix v.p. in $v_{k-1, 2^{k-1}}$. Now we see that ${\tau }_{i, 2^{k-2}}{{\tau }_{j, 2^{k-2}}}={{\tau }_{i, j}}$, where $1 \leq i,j < 2^{k-2}$.
To get ${\tau }_{m, l}$ from $v_{1,2}X^{[k-1]}$, i.e. $2^{k-2} < m,l \leq 2^{k-1} $ we use $\alpha_0$ to map ${\tau }_{i, j}$ in ${\tau }_{i+2^{k-2}, j+2^{k-2}}\in v_{1,2} AutX^{[k-1]}$. To construct arbitrary transposition ${\tau }_{i,m}$ from $W_{k-1}$ we have to multiply ${\tau }_{1,i} {\tau } {\tau }_{m,2^{k-1}}={\tau }_{i,m}$.
Let us realize a natural number of $v_{k,l}$, $1<l<2^k$ in 2-adic set of presentation (binary arithmetic).  Then $l={\delta_{{{l}_{1}}}}{{2}^{m_l}}+{\delta_{{{l}_{2}}}}{{2}^{m_l-1}}+...+{\delta_{{{l}_{m_l+1}}}},\,\, \delta_{l_i} \in \{0,1\}$ where is a correspondence between  ${\delta_{{{l}_{i}}}}$ that from such presentation and expressing of automorphisms: $\tau_{l,2^{k-1}} = \prod_{i=1}^{m_l} \alpha_{k-2-(m_{l}-i)}^{\delta_{{l}_{i}}} \tau  \prod_{i=1}^{m_l} \alpha_{k-2-(m_{l}-i)}^{\delta_{l_i}},  1 \leq m_l \leq k-2$.  In other words $\left\langle {{\alpha }_{0}},...,{{\alpha }_{k-2}},\tau  \right\rangle \simeq {{G}_{k}}$.
 \end{proof}

\begin{corollary}\label{genG_k-1} The elements from condition of Lemma \ref{gen} are enough to generate a basis of $W_{k-1}$.
\end{corollary}

 \begin{lemma}\label{ordG_k} Sizes of groups $G_k = \langle   \mathop{\alpha}_{0}, \mathop{\alpha}_{1},
 \mathop{\alpha}_{2},...,\mathop{\alpha}_{k-2}, \mathop{\tau} \rangle $
and $Syl_2(A_{2^{k}})$ are equal to $2^{2^{k}-2}$.
\end{lemma}
 \begin{proof}
In accordance with Legendre's formula, the power of 2 in ${{2}^{k}}!$ is $\left[ \frac{{{2}^{k}}}{2} \right]+\left[ \frac{{{2}^{k}}}{{{2}^{2}}} \right]+\left[ \frac{{{2}^{k}}}{{{2}^{3}}} \right]+...+\left[ \frac{{{2}^{k}}}{{{2}^{k}}} \right]=\frac{{{2}^{k}}-1}{2-1}$. We need to subtract 1 from it because we have only $\frac{n!} {2}$ of all permutations as a result: $\frac{{{2}^{k}}-1}{2-1}-1=2^{k}-2$. So $\left| Syl({{A}_{{{2}^{k}}}}) \right|={{2}^{{{2}^{k}}-2}}$.
The same size has group $G_k=B_{k-1} \ltimes W_{k-1}$ and $|G_k|=|B_{k-1}|\cdot|W_{k-1}|= |Syl_2 A_{2^k}|$. Since size of groups $G_{k}$ according to Proposition \ref{ordW} and the fact that $|B_{k-1}|=2^{2^{k-1}-1}$ is $2^{2^{k}-2}$.
For instance the sizes of $Syl_2 (A_8)$, $B_{3-1}$ and $W_{3-1}$ are such $|W_{3-1}|= 2^{2^{3-1}-1}=2^3=8$, $|B_{3-1}|=|C_2\wr C_2| = 2 \cdot 2^2=2^3$ and according to Legendre's formula, the power of 2 in ${{2}^{k}}!$ is $\frac{{2}^{3}}{2}+ \frac{{2}^{3}}{2^2}+\frac{{2}^{3}}{2^3} -1=6$ so $Syl_2 (A_8) = 2^6=2^{2^k-2}$, where $k=3$. Next example for $A_{16}$: $Syl_2 (A_{16}) =2^{2^4-2}= 2^{14}, k=4$, $|W_{4-1}|= 2^{2^{4-1}-1}=2^7$, $|B_{4-1}|=|C_2\wr C_2\wr C_2| = 2 \cdot 2^2\cdot 2^4 = 2^7$. So we have the $|A_{16}|=|W_{3}||B_{3}|$ equality which endorse the condition of this lemma. 
 \end{proof}

\begin{theorem}
\textbf{A maximal 2-subgroup}  of $Aut{{X}^{\left[ k \right]}}$ that acts by even permutations on ${{X}^{k}}$ has the structure of the semidirect product $G_k \simeq  B_{k-1} \ltimes W_{k-1} $ and isomorphic to $Syl_2A_{2^k}$.
\end{theorem}
\begin{proof}
A maximal 2-subgroup of  $Aut{{X}^{\left[ k-1 \right]}}$ is isomorphic to ${{B}_{k-1}}\simeq \underbrace{{{C}_{2}}\wr {{C}_{2}}\wr ...\wr {{C}_{2}}}_{k-1}$ (this group acts on ${{X}^{k-1}}$). A maximal 2-subgroup which has elements with active states only on ${{X}^{k-1}}$ corresponds subgroup ${W}_{k-1}$.
Since subgroups ${B}_{k-1}$ and ${{W}_{k-1}}$ are embedded in $Aut{{X}^{\left[ k \right]}}$, then define an action of ${{B}_{k-1}}$ on elements of ${{W}_{k-1}}$ as ${{\tau }^{\sigma }}=\sigma \tau {{\sigma }^{-1}},\,\,\,\sigma \in {{B}_{k-1}},\,\,\tau \in {{W}_{k-1}}$,
i.e. action by inner automorphism (inner action) from $Aut{{X}^{\left[ k \right]}}$.
Note that ${{W}_{k-1}}$ is subgroup of stabilizer of  ${{X}^{k-1}}$ i.e. ${{W}_{k-1}}<St_{Aut{X}^{[k]}}(k-1)\lhd AutX^{[k]}$ and is normal too $W_{k-1}\lhd AutX^{[k]}$, because conjugation keeps a cyclic structure of permutation so even permutation maps in even. Therefore such conjugation induce automorphism of ${W}_{k-1}$ and $G_k \simeq B_{k-1}\ltimes W_{k-1}$. Since  at ${{X}^{k-1}}$ is ${{2}^{{{2}^{k-1}}}}$ vertexes and half of combinations of active states from  ${X}^{k-1}$ can form even permutation thus $\left| {{W}_{k-1}} \right|={{2}^{{{2}^{k-1}}-1}}$ that is proved in Proposition \ref{ordW}. Using the Corollary \ref{B_k-1} about ${{B}_{k-1}}$ we get size of  ${{G}_{k}}\simeq {{B}_{k-1}} \ltimes {{W}_{k-1}}$ is ${{2}^{{{2}^{k-1}}-1}}\cdot {{2}^{{{2}^{k-1}}-1}}={{2}^{{{2}^{k}}-2}}$.
 Since $G_k$ is the maximal 2-subgroup then $G_k \simeq Syl_2A_{2^k}$.
\end{proof}

 \begin{theorem} \label{isomor}  
  The set $S_{\mathop{\alpha}}= \{\mathop{\alpha}_{0}, \mathop{\alpha}_{1},
 \mathop{\alpha}_{2},  ... ,\mathop{\alpha}_{k-2}, \mathop{\tau}\}$
   of elements from subgroup of $AutX^{[k]}$
   generates a group $G_k$ which isomorphic to $Syl_2(A_{2^{k}})$.
\end{theorem}
 \begin{proof}
As we see from Corollary \ref{B_k-1}, Lemma \ref{gen} and Corollary \ref{genG_k-1} group $G_k$ are generated by $S_{\mathop{\alpha}}$ and their sizes according to Lemma \ref{ordG_k} are equal. So according to Sylow's theorems 2-subgroup $G_k<A_{2^k}$ is $Syl_2 (A_{2^k})$.
 \end{proof}
 Consequently, we construct a generating set, which contains $k$ elements, that is less than in \cite{Iv}.

The structure of Sylow 2-subgroup of $A_{2^k}$ is the following: $\underset{i=1}{\overset{k-1}{\mathop{\wr }}}\,{{C}_{2}}  \ltimes \prod_{i=1}^{2^{k-1}-1} C_2  $, where we take $C_2$ as group of action on two elements and this action is faithful. It adjusts with construction of normalizer for $Syl_p(S_n)$ from \cite{Weisner}, where it was said that $Syl_2(A_{2^l})$ is self-normalized  in $S_{2^l}$.


\begin {definition}
Let us call the index of automorphism $\beta$ on $X^l$ as a number of active v.p. of $\beta$ on $X^l$.
\end {definition}

 \begin {definition} Define an \textbf{element of type }\texttt{T} as an automorphism ${{\tau }_{{{i}_{0}},...,{{i}_{{{2}^{k-1}}}};{{j}_{{{2}^{k-1}}}},...,{{j}_{{2}^k}}}}$, that has even index at ${{X}^{k-1}}$
and it has exactly $m_1$ active states, $ m_1\equiv 1 (mod 2)$, in vertexes ${{v}_{k-1,j}}, $ with number $1 \leq j \leq  2^{k-2}$ and $m_2, \, m_2\equiv 1 (mod 2)$ active states in vertices ${{v}_{k-1,l}}$, ${2^{k-2}} < l \leq {2^{k-1}}$. Set of such elements is denoted by \texttt{T}.
\end {definition}


\begin {definition} A combined element is such an automorphism $ {{\beta }_{1,{{i}_{1}};2,{{i}_{2}};...;k-1,{{i}_{k-1}}; \tilde{\tau }}}$, that it's restriction
${{\beta}_{1,{{i}_{1}};2,{{i}_{2}};...;k-1,{{i}_{k-1}};\tilde{\tau }}}\left|_{{{X}_{k-1}}} \right.$ coincide with one of elements that can be generated by ${{S}_{\alpha}}$ and $Rist_{<{{\beta}_{1,{{i}_{1}};2,{{i}_{2}};...;k-1,{{i}_{k-1}}; \tilde{ \tau} }}>}(k-1)  =\left\langle \tau' \right\rangle $ \cite{Ne} where $\tau' \in $\texttt{T}. Set of such elements is denoted by \texttt{C}.
\end {definition}
In other word elements $g \in$\texttt{C}  on level ${{X}^{k-1}}$ have such structure as element and generator of type \texttt{T}. As well ${{\tau }_{{{i}_{0}},...,{{i}_{{{2}^{k-1}}}};{{j}_{{{2}^{k-1}}}},...,{{j}_{{{2}^{k}}}}}}\in S{{t}_{Aut{{X}^{k}}}}(k-1)$.

The minimum size of a generating set  $S$
of G we denote by rk$G$ and call the rank of $G$ \cite{Bog}.
By distance between vertices we shall understand usual distance at graph between its vertexes.
By distance of automorphism $g$ (element) we shall understand maximal distance between two vertexes with active states of $g$.

\begin{lemma} \label{Lemma about keeping of distance} A vertices permutations on ${X}^{k}$ that has distance ${{d}_{0}}$ can not be generated by vertex permutations with distance ${{d}_{1}}$ such that \, ${{d}_{1}}<{{d}_{0}}$.
\end{lemma}
\begin{proof}
 The element ${{\tau }_{ij}}$ with distance $\rho ({{\tau }_{ij}})={{d}_{0}},\,\,\,{{d}_{0}}<{{d}_{1}}$ can be mapped by automorphic mapping only in automorphism with distance ${{d}_{0}}$ because automorphic mapping keep incidence relation and so it possess property of isometry. Also multiplication of portrait (labeled graph) of automorphism ${{\tau }_{ij}}:\,\,\,\,\,\rho ({{\tau }_{ij}})={{d}_{1}}$ give us portrait of element with distance no greater than ${{d}_{1}}$, it follows from properties of group operation. For instance $\tau_{1i} \tau_{1j}=\tau_{ii}$, where $i, j > 2^{k-2}$ $\rho ({{\tau }_{1i}})=\rho ({{\tau }_{1j}})=2k-2$ but $\rho ({{\tau }_{ij}})< 2k-2$.
\end{proof}

\begin{lemma} \label{about transposition} An arbitrary automorphism $\tau{'} \in$\texttt{T} (or in particular $\tau $) can be generated only with using odd number of automorphisms from \texttt{C} or \texttt{T}.
\end{lemma}
\begin{proof}
Let us assume that there is no such element ${{\tau }_{ij}}$ which has distance $2k-2$ then accord to Lemma \ref{Lemma about keeping of distance} it is imposable to generate are pair of transpositions $\tau{'} $ with distance $\rho ({{\tau }_{ij}})=2k$ since such transpositions can be generated only by ${{\tau }_{ij}}$ that described in the conditions of this Lemma: $i\le {{2}^{k-2}},\,\,\,j>{{2}^{k-2}}$. Combined element can be decomposed in product $ \tau \dot {\beta }_{{i}_{l}} = {{\beta }_{{{i}_{l}};\tau }} $ so we can express by using $\tau$ or using a product where odd number an elements from \texttt{T} or \texttt{C}. If we consider product $P$ of even number elements from \texttt{T} then automorphism $P$ has even number of active states in vertexes ${{v}_{k-1,i}}$ with number $ i \leq  2^{k-2}$ so $P$ does not satisfy the definition of generator of type \texttt{T}.
\end{proof}
\begin{corollary} \label{About generating distance} Any element of type \texttt{T} can not be generated by ${{\tau }_{ij}}\in Aut {{v}_{1,1}}{{X}^{[k-1]}}$ and ${{\tau }_{ml}}\in Aut {{v}_{1,2}}{{X}^{[k-1]}}$. The same corollary is true for a combined element.
\end{corollary}
\begin{proof}
It can be obtained from the Lemma \ref{Lemma about keeping of distance} because such ${\tau_{ij}}\in Aut {{v}_{1,1}}{{X}^{[k-1]}}$ has distance less then $2k-2$ so it does not satisfy conditions of the Lemma \ref{Lemma about keeping of distance}. I.e. $\tau{'}$ can not be generated by vertices automorphisms with distance between vertices less than $2k-2$ such distance has only automorphisms of type \texttt{T} and \texttt{C}. But elements from $ Aut {{v}_{1,1}}{{X}^{[k-1]}}$ do not belongs to type \texttt{T} or \texttt{C}.  
\end{proof}


\begin{lemma} \label{About not closed set of element of type T} Sets of elements of types \texttt{T}, \texttt{C}  are not closed by multiplication and raising to even power.
\end{lemma}

\begin{proof}
Let  $\varrho, \rho \in$ \texttt{T} (or \texttt{C}) and $\varrho, \rho = \eta$. The numbers of active states from $\varrho$ and $\rho$ in tuple of vertices $v_{k-1, i}$, $1 \leq i \leq 2^{k-2}$ sums by $mod 2$, numbers of active states from $\varrho$ and $\rho$ in vertices on vertices $v_{k-1, i}$, $2^{k-2} <  i \leq 2^{k-1}$ sums by $mod 2$ too. Thus $\eta$ has even numbers of active states on these tuples.
 Hence $ RiS{{t}_{\left\langle \eta  \right\rangle }}(k-1)$ doesn't contain elements of type $\tau $ so $\eta \notin $\texttt{T}. Really if we raise the element ${{\beta }_{1,{{i}_{1}};2,{{i}_{2}};...;k-1,{{i}_{k-1}};\tau }}\in $\texttt{T} to even power or we evaluate a product of even number of multipliers from \texttt{C} corteges ${{\mu }_{0}}$ and ${{\mu }_{1}}$ permutes with whole subtrees ${{v}_{1,1}}{{X}^{[k-1]}}$ and ${{v}_{1,2}}{{X}^{[k-1]}}$, then we get an element $g$ with even indices of ${{X}^{k}}$ in $v_{1,1}{{X^{k-1}}}$ and $v_{1,2}{{X^{k-1}}}$. Thus $g \notin $\texttt{T}. Consequently elements of \texttt{C} do not form a group, and the set \texttt{T} as a subset of \texttt{C} is not closed too.
\end{proof}

Let ${S^{'}_{\alpha }}=\left\langle {{\alpha }_{0}},\,{{\alpha }_{1}},...,{{\alpha }_{k-2}} \right\rangle $ so as well known \cite{Gr} $\left\langle S_{\alpha }^{'} \right\rangle =Aut{{X}^{[k-1]}}$. The cardinality of a generating set $S$ is denoted by $\mid {{S}}  \mid$ so $\mid {{S}^{'}_{\alpha}}  \mid=k-1$. Recall that $rk\left( G \right)$ is the rank of a group $G$ \cite{Bog}.


Let $S_{\beta }={{S}_{\alpha }^{'}}\cup \tau_{i...j} $, where $\tau_{i...j} \in$\texttt{T} and
${{S}_{\beta }^{'}}$ is generating system which contains combine elements, $\mid {{S}^{'}_{\beta}} \mid = k$.
It's known that $rk(Aut{{X}^{[k-1]}})= k-1 $ and $\mid {{S}^{'}_{\alpha}} \mid = k-1$ \cite{Gr}. So if we complete ${S}^{'}_{\alpha}$ by $\tau$ or element of type \texttt{T} we obtain set ${{S}_{\beta }}$ such that ${{G}_{k}}\simeq \left\langle {{S}_{\beta }} \right\rangle $ and $ |{S}_{\beta }|= k$.
Hence to construct combined element $\beta $ we multiply generator ${{\alpha}_{i}}$ of ${{S}^{'}_{\alpha}}$ or arbitrary element that can be express from ${{S}^{'}_{\alpha}}$ on the element of type \texttt{T}, i.e., we take $\tau' \cdot {{\beta }_{i}}$ instead of ${{\beta }_{i}}$ and denote it ${\beta }_{i; \tau'}$. It's equivalent that $Rist_{{\beta }_{i; \tau'} }(k-1)=\left\langle \tau' \right\rangle $, where $\tau' \in$\texttt{T}.

Let us assume that $ S_{\beta }^{'}$ has a cardinality $k-1$. If in this case $ S_{\beta }^{'}$ is generating system again, then element $\tau $ can be expressed from it. There exist too ways to express element of type \texttt{T} from ${{S}_{\beta }^{'}}$.
 To express element of type \texttt{T} from ${{S}_{\beta }^{'}}$ we can use a word ${{\beta }_{i,\tau }}\beta _{i}^{-1}=\tau $ but if ${{\beta }_{i,\tau }}\in {{S}_{\beta }^{'}}$ then ${{\beta }_{{{i}}}}\notin S_{\beta }^{'}$ in contrary case $\mid S_{\beta }^{'} \mid = k$.
 So we can not express word ${\beta }_{i,\tau }\beta _{i}^{-1}\left| _{{X}^{[k-1]}} \right.=e$ to get ${{\beta }_{i,\tau }}\beta _{i}^{-1}=\tau $.
      For this goal we have to find relation in a group that is a restriction of the group ${{G}_{k}}$ on ${{X}^{[k-1]}}$. We have to take in consideration that $ {{G}_{k}}\left| _{{{X}^{[k-1]}}} \right.={{B}_{k-1}} $.
             Really in wreath product $\wr _{j=1}^{k}{{\mathcal{C}}^{(j)}_{2}}\simeq B_{k-1}$ holds a constitutive relations
      $\alpha_{i}^{2^m}=e$ and
  $\left[ \alpha _{m}^{i}{{\alpha }_{{{i}_{n}} }}\alpha_{m}^{-i},\,\alpha_{m}^{j}{\alpha }_{{{i}_{k}} }\alpha_{m}^{-j} \right]=e,\,\,\,i\ne j$, where $\alpha_{m} \in {{S}^{'}_{^{\alpha }}}$, $\alpha_{i_k} \in {{S}^{'}_{^{\alpha }}}$ are generators of factors of $\wr _{j=1}^{k}{{\mathcal{C}}^{(j)}_{2}}$ ($m<n$, $m<k$) \cite{Sk, DrSku}.
    Such relations are a words $\left[ \beta _{m}^{i}{{\beta }_{{{i}_{n}},{\pi} }}\beta_{m}^{-i},\,\beta _{m}^{j}{\beta }_{{{i}_{k}}, {\pi} }\beta_{m}^{-j} \right],\,\,\,i\ne j$ or $\beta_{i}^{2^m}=e$,
   ${\beta }_{{{i}_{n}}}$, ${\beta_m},\,\,{{\beta }_{{{i}_{k}}}},\, {{\beta }_{{{i}_{n}}, \pi }}$ are generators of
    $S^{*}_{\beta}(k-1)$, that could be an automorphism $ \theta$. But $\left[ \beta _{m}^{i}{{\beta }_{{{i}_{n}},{\pi} }}\beta_{m}^{-i},\,\beta _{m}^{j}{\beta }_{{{i}_{k}}, {\pi} }\beta_{m}^{-j} \right],\,\,\,i\ne j$ does not belongs to \texttt{T} because this word has logarithm 0 by every element \cite{K}. According to Lemma \ref{about transposition} and Lemma \ref{About not closed set of element of type T} product of even number element of type C doesn't equal to element of C or T.


Let us assume an existence of generating set of cardinality $k-1$ for $Syl_2(A_{2^k})$ that in gene\-ral case has form $S_{\beta }^{*}(k-1)=\left\langle
 \mathop {\beta}_0, \mathop {\beta}_{0;1,(i_{11},i_{12}); \pi_1 } , ... ,    \beta_{0;...;k-1,(i_{k-11},...,i_{k-1j});\pi_{k-1} }
   \right\rangle, \\ 0<i<k-1, j \leq 2^{i}, \pi_1 \in T $, where $\pi_i$ is the cortege of vertices from $X^{k-1}$ with non trivial v.p. which realize permutation with distance $2(k-1)$. In other word if element $\pi_i \in$ \texttt{T} then $\beta_{0;...;i,(i_{i1},...,i_{ij}); \pi_i } = \pi_i \alpha_{0;...;i,(i_{i1},...,i_{ij}) }$.
    From it follows $\mathop {\beta}_{0;1,(i_{11},i_{12}); ... ;{m},(i_{m1},...,i_{mj};\pi_m) }\mid_{X^{[k-1]}} = \mathop {\alpha}_{0;1,(i_{11},i_{12}); ... ,{m},(i_{m1},...,i_{mj}) } \in S^{'}_{\alpha}$.

Note, that automorphisms from set $S^{*}_{\beta}(k-1)$ generate on truncated rooted tree \cite{Ne} $X^{[k-1]}$ group $\langle S^{*}_{\beta}(k-1) \rangle \mid_{X^{[k-1]}} \simeq \wr^{k-1}_{i=1} C_2^{(i)}\simeq B_{k}$.
\begin{theorem} \label{Th about general relation}
Any element of type \texttt{T} can not be expressed by elements of $S^{*}_{\beta}(k-1)$.
\end{theorem}
\begin{proof}
It is necessary to express an automorphism $ \theta$ of type \texttt{T} express such automorphism which has zero indexes of $X^0, ... , X^{k-2}$, this conclusion follows from structure of elements from \texttt{T}. It means that word $w$ from letters of $S^{*}_{\beta}(k-1)$ such that $w = \theta$ is trivial in group $B_k$, that arise on restriction of $\langle S^{*}_{\beta}(k-1)\rangle$ on $X^{[k-1]}$, as well restriction $ G_k \mid_{X^{[k-1]}}\simeq B_k$.

 Every relation from $B_{k}$ can be expressed as a product of words from the normal closure $R^{B_{k}}$ of the set of constitutive relations of the group $B_{k}$ \cite{Bog}.
    But defined relations $r_i$ of $B_{k}$ have form of commutators \cite{DrSku, Sk}
  so the number of inclusions of every multiplier is even and as follows from lemma \ref{About not closed set of element of type T}
that $r_i \mid_{[X^k]} \notin$ \texttt{T}.
             Really in wreath product $\wr _{j=1}^{k}{{\mathcal{C}}^{(j)}_{2}}\simeq B_{k-1}$ holds a constitutive relations
      $\alpha_{i}^{2^m}=e$ and
  $\left[ \alpha _{m}^{i}{{\alpha }_{{{i}_{n}} }}\alpha_{m}^{-i},\,\alpha_{m}^{j}{\alpha }_{{{i}_{k}} }\alpha_{m}^{-j} \right]=e,\,\,\,i\ne j$, where $n,k,m$ are number of groups in $\wr _{j=1}^{k}{{\mathcal{C}}^{(j)}_{l}}$ ($m<n$, $m<k$) \cite{Sk, DrSku}, where
 ${\alpha }_{{{i}_{n}}}$, ${\alpha }_{{{i}_{n}}}$, ${{\alpha }_{m}},\,\,{{\alpha }_{{{i}_{k}}}}$ are generators of the $B_{k-1}$ from ${{S}^{'}_{^{\alpha }}}$.
    So it give us a word $\left[ \beta _{m}^{i}{{\beta }_{{{i}_{n}},{\pi} }}\beta_{m}^{-i},\,\beta _{m}^{j}{\beta }_{{{i}_{k}}, {\pi} }\beta_{m}^{-j} \right],\,\,\,i\ne j$ that could be an automorphism $ \theta$ but does not belongs to \texttt{T} because the word has structure of commutator or belongs to normal closure $R^{B_{k-1}}$ so has logarithm 0 by every element, where
   ${\beta }_{{{i}_{n}}}$, ${\beta_m},\,\,{{\beta }_{{{i}_{k}}}},\, {{\beta }_{{{i}_{n}}, \pi }}$ are generators of $G_k$ from
    $S^{*}_{\beta}(k-1)$.

    Let ${{\beta}_{1,1,(i_{11},i_{12}), \ldots ,i,(i_{i1},...,i_{i2^i});\pi_i}}$ be an arbitrary element of type \texttt{C}, where $\pi_i$ is cortege of vertexes from $X^{k-1}$ having non trivial states which realize permutation with distance $2(k-1)$.
 The case where $\theta = \beta_{{k-1}; {\pi_i}} \in$ \texttt{T} can be expressed by multiplying arbitrary \\ ${{\beta}_{1,(i_{11,12}), \ldots ,i,(i_{i1},...,i_{i2^i});\pi_i}}$ on $\beta_{1,(i_{11,12}), \ldots ,i,(i_{i1},...,i_{i2^i})}^{-1}$ means that such set has size more then $k-1$ what is contradiction.
   Really if ${{\beta}_{1,(i_{11,12}), \ldots ,i,(i_{i1},...,i_{i2^i})}}$ and ${{\beta}_{1,(i_{11,12}), \ldots ,i,(i_{i1},...,i_{i2^i});\pi_i}} \,  \in S_{\beta}^{*}$ then ${\beta}_{1,(i_{11,12}), \ldots ,i,(i_{i1},...,i_{i2^i})}  |_{X^{[k-1]}} = {{\beta}_{1,(i_{11,12}), \ldots ,i,(i_{i1},...,i_{i2^i});\pi_i}}|_{X^{[k-1]}} $ i.e. these elements are mutually inverse at restriction on $X^{[k-1]}$,
    but it means that in restriction of $S^{*}_{\beta}(k-1)$ to $X^{[k-1]}$ that corresponds to the generating set $S^{'}_{\alpha}(k-1)$ for $AutX^{[k-1]} \simeq B_{k-1} $ we use two equal generators. So it has at least $k$ generators,
    because $rk(Aut X^{[k-1]})=k-1$ according to lemma \ref{B_k-1}.

The subcase of this case where ${{\beta}^{-1}_{1,(i_{11},i_{12}), \ldots ,l,(i_{l1},...,i_{l2^i})}}$ 
 can be expressed from $S^{*}_{\beta}(k-1)$ as a product of its generators has the same conclusion. Really if we can generate arbitrary element from $B_{k-1}$ by generators from $S^{*}_{\beta}(k-1)$ then $k-1$ generators is contained in set $S^{*}_{\beta}(k-1)$ but we have a further ${{\beta}_{1,(i_{11},i_{12}), \ldots ,i,(i_{i1},...,i_{i2^i}); {\pi}}}$. In other words if arbitrary element ${{\beta}_{1,(i_{11},i_{12}), \ldots ,l,(i_{l1},...,i_{l2^i})}}$ of the $B_{k-1}$ does not contains in $S^{*}_{\beta}(k-1)$ but can be expressed from it then $S^{*}_{\beta}(k-1)$ has at least $k-1$ elements exclusive of ${{\beta}_{1,(i_{11},i_{12}), \ldots ,i,(i_{i1},...,i_{i2^i}); {\pi}}}$  then $S^{*}_{\beta}(k-1) \geq k $.
\end{proof}
\begin{corollary}\label{generating pair} A necessary and sufficient condition of expressing an element  $\tau $ from $S_{\beta }^{'}$  is existing of pair:  ${{\beta }_{{{i}_{m}};\tau }},\,\,\,\,{{\beta }_{{{i}_{_{m}}}}}$ in $\langle S_{\beta}^{'} \rangle$. So size of a generating set of $G_k$ which contains $S_{\beta }^{'}$ is at least $k$.
\end{corollary}
\begin{proof} Proof can be obtained from Theorem \ref{Th about general relation} and
Lemma \ref{Th about general relation} from which we have that element of type \texttt{T} cannot be expressed from $\left[ \beta _{m}^{i}{{\beta }_{{{i}_{n}}.\tau }}\beta _{m}^{-i},\,\beta _{m}^{j}{{\beta }_{{{i}_{k}}.\tau }}\beta _{m}^{-j} \right]=e,\,\,\,i\ne j$ because such word has even number of elements from \texttt{C}. Sufficient condition follows from  formula ${{\beta }_{{{i}_{m}};\tau }}\beta _{{{i}_{_{m}}}}^{-1}=\tau $.
\end{proof}
\begin{lemma} \label{rk} A generating set of  ${{G}_{k}}$ contains $S_{\alpha}^{'}$ and has at least $k-1$ generators.
\end{lemma}
\begin{proof} The subgroup ${{B}_{k-1}}<{{G}_{k}}$ is isomorphic to $Aut{{X}^{k-1}}$ that has a  minimal set of generators  of $k-1$ elements \cite {Gr}. Moreover, the subgroup ${{B}_{k-1}} \simeq {}^{{{G}_{k}}}/{}_{{{W}_{k-1}}}$, because $G_{k} \simeq B_{k-1}\ltimes {W}_{k-1}  $, where ${W}_{k-1} \vartriangleright G_k$. As it is well known that if $H\lhd G$ then $\text{rk}(G)\ge \text{rk}(H)$, because all generators of $G_{k}$ may belongs to different quotient classes \cite{Magn}.
\end{proof}
As a corollary of last lemma and Theorem \ref{Th about general relation} we see that generating set of size $k-1$ does not exist.



Let us sharpen and reformulate following theorem which is in \cite{Meld}.
\begin{theorem} \label{form of comm} An element of form
$(r_1, \ldots, r_{p-1}, r_p) \in W' = (C_p \wr B)'$ iff product of all $r_i$ (in any order) belongs to $B'$ and $cw(B)=1$.
\end{theorem}
We deduce here a special form of commutator elements.
If we multiply elements having form of a tuple $(r_1, \ldots, r_{p-1}, r_p)$, where $r_i={{h}_{i}}{{g}_{a(i)}}h_{ab(i)}^{-1}g_{ab{{a}^{-1}}(i)}^{-1}$, $h, \, g \in B$ and $a,b \in C_p$, then we obtain a product
\begin{equation}\label{Meld}
 \stackrel{p}{ \underset{\text{\it i=1}} \prod} r_i = \prod\limits_{i=1}^{p}{{{h}_{i}}{{g}_{a(i)}}h_{ab(i)}^{-1}g_{ab{{a}^{-1}}(i)}^{-1} \in B'}.
\end{equation}

Recall that $  \stackrel{k}{ \underset{\text{\it i=1}}{\wr }}C_p  \simeq Syl_p S_{p^k} $.
\begin{lemma} \label{comm B_k old}
 An element $g \in B_k \simeq  \stackrel{k}{ \underset{\text{\it i=1}}{\wr }}C_p  $ belongs to commutator subgroup $B'_k \simeq (\stackrel{k}{ \underset{\text{\it i=1}}{\wr }}C_p)'$ iff $g$ has all possible even indexes on ${{X}^{l}},\,\,l<k$.
 \end{lemma}

\begin{proof}
Let us prove the ampleness by induction by a number of level $l$.
 We first show that our statement for base of the induction is true. Actually, if $\alpha, \beta \in B_{0}$ then $(\alpha \beta {{\alpha }^{-1}}) \beta^{-1}$ determine a trivial v.p. on $X^{0}$. If $\alpha, \beta\in B_{1}$ and $\beta$ has an odd index on $X^1$, then $(\alpha \beta {{\alpha }^{-1}})$ and $\beta^{-1}$ have the same index on $X^1$. Consequently, in this case an index of the product $(\alpha \beta {{\alpha }^{-1}}) \beta^{-1}$ can be 0 or 2. Case where $\alpha, \beta \in B_{1}$ and has even index on $X^1$, needs no proof, because the product and the sum of even numbers is an even number.

To finish the proof it suffices to assume that for $B_{l-1}$ statement holds and prove that it holds for $B_l$.
Let $\alpha, \beta$ are an arbitrary automorphisms from $Aut X^{[k]}$ and $\beta$ has index $x$ on $X^{l}, \, l<k$, where $0 \leq x \leq 2^{l} $.
   A conjugation of an automorphism $\beta $  by arbitrary $\alpha \in Aut{{X}^{[k]}}$ gives us arbitrary permutations of $X^l$ where $\beta $ has active v.p.

 Thus following product $(\alpha \beta {{\alpha }^{-1}}) \beta ^{-1}$ admits all possible even indexes on $X^l,  l<k$ from 0 to $2x$. In addition $[\alpha, \beta]$ can has arbitrary assignment (arrangement) of v.p. on $X^l$. Let us present $B_k$ as $B_k=B_l \wr B_{k-l}$, so elements $\alpha, \beta$ can be presented in form of wreath recursion $\alpha = (h_{1},...,h_{2^l })\pi_1, \,  \beta = (f_{1},...,f_{2^l })\pi_2$, $h_{i}, f_{i} \in B_{k-l} ,\  0<i \leq 2^l$ and $h_{i}, f_{j}$ corresponds to sections of automorphism in vertices of $X^l$ of isomorphic subgroup to $B_l$ in $Aut X^{[k]}$.
Actually, the parity of this index are formed independently of the action of
$Aut X^{[l]}$ on $X^l$. So this index forms as a result of multiplying of elements of commutator presented as wreath recursion $(\alpha \beta \alpha^{-1}) \cdot \beta ^{-1} = (h_{1},...,h_{2^l})\pi_1 \cdot (f_{1},...,f_{2^l })\pi_2= (h_{1},...,h_{2^l}) (f_{\pi_1(1)},...,f_{\pi_1(2^l)})\pi_1 \pi_2 $, where $h_{i}, f_{j} \in {B}_{k-l}$, $l<k$ and besides automorphisms corresponding to $h_{i}$ are $x$ automorphisms which has active v.p. on $X^l$. Analogous automorphisms $h_{i}$ has number of active v.p. equal to $x$. As a result of multiplication we have automorphism with index $2i:$ $0 \leq 2i \leq 2x$.
Consequently, commutator $[\alpha, \beta]$ has arbitrary even indexes on $X^m$, $m<l$ and we showed by induction  that it has even index on $X^l$.

Let us prove the necessity by induction by number of level $l$. Let $g \in B_k$ and $g$ has all even indexes on $X^j$ $0 \leq j < k$ we need to show that $g \in B'_k$. According to condition of this Lemma $g_1 g_2$ has even indexes. An element $g$ has form $g=(g_1, g_2)$, where $g_1, g_2 \in B_{k-1}$, and products $g_1 g_2 = h, g_2 g_1= h' \in B'_{k-1}$ because $h, h' \in B_{k-1}$ and for $B_{k-1}$ induction assumption  holds. Therefore all products of form $g_1 g_2$ indicated in formula \ref{Meld} belongs to $B'_{k-1}$. Hence, from Theorem \ref{form of comm} follows that $g= (g_1, g_2 ) \in B'_k$.
\end{proof}

An even easier Proposition, that needs no proof, is the following.
\begin{proposition}\label{B_k_criteria}
An element $(g_1, g_2)\sigma^i$, $i \in \{0, 1\} $ of wreath power $\stackrel{k}{ \underset{\text{\it i=1}}{\wr } }C_2  $ belongs to its subgroup $ G_k$, iff $g_1g_2 \in G_{k-1}$.
\end{proposition}

\begin{proposition}\label{B'_k and B^2_k}
If $g $ is an element of wreath power $\stackrel{k}{ \underset{\text{\it i=1}}{\wr } }C_2 \simeq B_k $ then $g^2 \in B'_{k}$.
\end{proposition}
\begin{proof}
As it was proved in Lemma \ref{comm B_k old} commutator $[\alpha, \beta]$ from $B_k$ has arbitrary even indexes on $X^m$, $m<k$. Let us show that elements of $B_k^2$ have the same structure.

Let $\alpha, \beta \in B_k$ an index of the automorphisms $\alpha^2 $, $(\alpha \beta)^2 $ and $\alpha, \beta  \in G_k$ on $X^l, \, l<k-1$ are always even. In more detail the indexes of $\alpha^2 $, $(\alpha \beta)^2 $ and $\alpha^{-2} $ on $X^l$ are determined exceptionally by the parity of indexes of $\alpha $ and $\beta $ on ${{X}^{l}}$. Actually, the parity of this index are formed independently of the action of
$Aut X^l$ on $X^l$.
So this index forms as a result of multiplying of elements $\alpha \in B_k$ presented as wreath recursion $ \alpha^2 = (h_{1},...,h_{2^l})\pi_1 \cdot (f_{1},...,f_{2^l })\pi_1= (h_{1},...,h_{2^l}) (f_{\pi_1(1)},...,f_{\pi_1(2^l)})\pi_1 \pi_1 $, where $h_{i}, f_{j} \in {B}_{k-l}, \, \pi_1 \in B_l$, $l<k$ and besides automorphisms corresponding to $h_{i}$ are $x$ automorphisms which has active v.p. on $X^l$. Analogous automorphisms $h_{i}$ has number of active v.p. equal to $x$. As a result of multiplication we have automorphism with index $2i:$ $0 \leq 2i \leq 2x$.

Since $g^2 $ admits only an even index on $X^l$ of $Aut X^{[k]}$, $0<l<k$, then $g^2 \in B'_{k}$ according to lemma \ref{comm B_k old} about structure of a commutator subgroup.
\end{proof}

\begin{lemma}\label{L_k_comm_criteria}
An element $(g_1, g_2)\sigma^i \in G_k'$ iff $g_1,g_2 \in G_{k-1}$ and $g_1g_2\in B_{k-1}'$.
\end{lemma}

\begin{proof}
Indeed, if $(g_1, g_2) \in G_k'$ then indexes of $g_1$ and $g_2$ on $X^{k-1}$ are even according to Lemma \ref{comm} thus, $g_1,g_2 \in G_{k-1}$. A sum of indexes of $g_1$ and $g_2$ on $X^{l}$, $l<k-1$ are even according to Lemma \ref{comm} too, so index of product $g_1 g_2$ on $X^{l}$ is even. Thus, $g_1g_2\in B_{k-1}'$.

Let us prove the sufficiency via Lemma \ref{comm}.
Wise versa, if $g_1,g_2 \in G_{k-1}$ then indexes of these automorphisms on $X^{k-2}$ of subtrees $v_{11}X^{[k-1]}$ and $v_{12}X^{[k-1]}$ are even as elements from $ G_k'$ have.  
 The product $g_1g_2$ belongs to $B_{k-1}'$ by condition of this Lemma so sum of indexes of $g_1, g_2$ on any level $X^l$,  $0 \leq l<k-1 $ is even. Thus, the characteristic properties of $G_k'$ described in this Lemma \ref{comm} holds.
\end{proof}

Let $X_1=\{v_{k-1,1}, v_{k-1,2},..., v_{k-1,2^{k-2}} \} $ and $X_2=\{v_{k-1,2^{k-2}+1}, ..., v_{k-1,2^{k-1}} \}$.
Let group $Sy{{l}_{2}}{{A}_{{{2}^{k}}}}$ acts on $X^{[k]}$.
\begin{lemma} \label{comm}
An element $g$ belongs to $G_k' \simeq (Syl_2{A_{2^k}})'$ iff $g$ is arbitrary element from $G_k$ which has all even indexes.



 \end{lemma}

\begin{proof}

Let us prove the ampleness by induction by a number of level $l$.
Recall that any authomorphism $\theta \in Syl_2 A_n$ has an even index on $X^{k-1}$ so number parities of active v. p. on ${{X}_{1}}$ and on ${{X}_{2}}$ are the same.
Conjugation by automorphism $\alpha$ from $Aut{{v}_{11}}{{X}^{\left[ k-1 \right]}}$ of automorphism $\theta $, that has some number $x:$ $1 \leq x \leq 2^{k-2}$ of active v. p. on ${{X}_{1}}$ does not change $x$. Also automorphism $\theta^{-1} $ has the same number $x$ of v. p. on $X_{k-1}$ as $\theta $ has. If $\alpha$ from $Aut{{v}_{11}}{{X}^{\left[ k-1 \right]}}$ and $ \alpha \notin Aut{{X}^{\left[ k \right]}}$ then conjugation $(\alpha \theta {{\alpha }^{-1}})$ permutes vertices only inside $X_1$ ($X_2$).

 Thus, ${\alpha }\theta {\alpha^{-1} }$ and $\theta$ have the same parities of number of active v.p. on $X_1$ ($X_2$). Hence, a product ${\alpha }\theta {\alpha^{-1} } \theta^{-1}$ has an even number of active v.p. on $X_1$ ($X_2$) in this case. More over  a coordinate-wise sum by \texttt{mod2} 
  of active v. p. from $(\alpha \theta {{\alpha }^{-1}})$ and $\theta^{-1}$ on $X_1$ ($X_2$) is even and equal to $y:$ $0 \leq y \leq 2x$.


   If conjugation by $\alpha$ permutes sets $X_1$ and $X_2$ then there are  coordinate-wise sums of no trivial v.p. from $\alpha \theta \alpha^{-1} \theta^{-1}$ on $X_1$ (analogously on $X_2$) have form: \\ $({{s}_{k-1,1}}(\alpha \theta {{\alpha }^{-1}}),..., {{s}_{k-1, 2^{k-2}}}(\alpha \theta {{\alpha }^{-1}}) )\oplus ({{s}_ {k-1,1}}(\theta^{-1}), ..., {{s}_{k-1,{{2}^{k-2}}}}(\theta^{-1} ))$.
   This sum has even number of v.p. on $X_1$ and $X_2$ because $(\alpha \theta {{\alpha }^{-1}})$ and ${{\theta }^{-1}}$ have a same parity of no trivial v.p. on $X_1$ ($X_2$).  Hence, $(\alpha \theta {{\alpha }^{-1}}){{\theta }^{-1}}$ has even number of v.p. on ${{X}_{1}}$ as well as on ${{X}_{2}}$.


An authomorphism $\theta $ from $G_k$ was arbitrary so number  of active v.p. $x$ on $X_1$ is arbitrary. And ${\alpha }$ is arbitrary from $AutX^{[k-1]}$ so vertices can be permuted in such way that the commutator $[{\alpha },\theta]$ has arbitrary even number $y$ of active v.p. on $X_1$, $0 \leq y \leq 2x$.

 A conjugation of an automorphism $\theta $ having arbitrary index $x$, $1 \leq x \leq 2^{l}$ on ${{X}^{l}}$  by different $\alpha \in Aut{{X}^{[k]}}$  gives us all permutations of active v.p. that $\theta $ has on ${{X}^{l}}$.
 So multiplication $(\alpha \theta {{\alpha }^{-1}})\theta $ generates a commutator having index $y$  equal to coordinate-wise sum by $mod 2$ of no trivial v.p. from vectors $({{s}_{l1}}(\alpha \theta {{\alpha }^{-1}}),{{s}_{l}}_{2}(\alpha \theta {{\alpha }^{-1}}),...,{{s}_{l{{2}^{l}}}}(\alpha \theta {{\alpha }^{-1}}))\oplus ({{s}_{l1}}(\theta ),{{s}_{l}}_{2}(\theta ),...,{{s}_{l{{2}^{l}}}}(\theta ))$  on ${{X}^{l}}$. A indexes parities of  $\alpha \theta {{\alpha }^{-1}}$  and  ${{\theta }^{-1}}$ are same so their sum by $mod 2$ are even.  Choosing $\theta $ we can  choose an arbitrary index $x$ of $\theta $ also we can choose arbitrary $\alpha $ to make a permutation of active v.p. on ${{X}^{l}}$.  Thus, we obtain an element with arbitrary even index on ${{X}^{l}}$ and arbitrary location of active v.p. on ${{X}^{l}}$.

Check that property of number parity of v.p. on ${{X}_{1}}$  and on ${{X}_{2}}$  is closed with respect to conjugation. We know that numbers of active v. p. on ${{X}_{1}}$ as well as on ${{X}_{2}}$ have the same parities. So
action by conjugation only can permutes it, hence, we again get the same  structure of element. Conjugation by automorphism $\alpha $ from  $Aut{{v}_{11}}{{X}^{\left[ k-1 \right]}}$  automorphism $\theta $, that has odd number of  active v. p. on ${{X}_{1}}$  does not change its parity.
Choosing the $\theta $ we can choose arbitrary index $x$ of
$\theta $ on ${{X}^{k-1}}$ and number of active v.p. on ${{X}_{1}}$  and  ${{X}_{2}}$  also we can choose arbitrary $\alpha $ to make a permutation active v.p. on ${{X}_{1}}$  and  ${{X}_{2}}$. Thus, we can generate all possible elements from a commutant. Also this result follows from Lemmas \ref{L_k_comm_criteria} and \ref{comm B_k old}.

Let us check that the set of all commutators $K$ from $Syl_2 A_{2^k}$ is closed with respect  to multiplication of commutators. Let $\kappa_1, \kappa_2 \in K$ then $\kappa_1 \kappa_2$ has an even index on $X^l$, $l<k-1$ because  coordinate-wise sum $({{s}_{l,1}}(\kappa_1),..., {{s}_{k-1, 2^l}}(\kappa_1) )\oplus ({{s}_ {l,\kappa_1(1)}}(\kappa_2), ..., {{s}_{l,\kappa_1({{2}^{l}})}}(\kappa_2 ))$.
 of two $2^l$-tuples of v.p. with an even number of no trivial coordinate has even number of such coordinate.  Note that conjugation of $\kappa $ can permute sets ${{X}_{1}}$ and ${{X}_{2}}$  so parities of $x_1$ and $X_2$ coincide. It is obviously index of $\alpha \kappa \alpha^{-1}$ is even as well as index of $\kappa $.

Check that a set $K$ is a set  closed with respect  to conjugation.

 Let $\kappa \in K$, then $\alpha \kappa {{\alpha }^{-1}}$  also belongs to $K$, it is so because  conjugation does not change index of an automorphism on a level. Conjugation only  permutes vertices on level because elements of $Aut{{X}^{\left[ l-1 \right]}}$ acts  on vertices of  ${{X}^{l}}$. But as it was proved above elements  of $K$ have all possible indexes on ${{X}^{l}}$, so as a result of conjugation $\alpha \kappa {{\alpha }^{-1}}$ we obtain an element from $K$.

Check that the set of commutators is closed with respect to multiplication of commutators.
Let $\kappa_1, \kappa_2 $ be an arbitrary commutators of $G_k$. The parity of the number of vertex permutations on $X^l$ in the product $\kappa_1 \kappa_2 $  is determined exceptionally by the parity of the numbers of active v.p. on ${{X}^{l}}$ in $\kappa_1$ and $\kappa_2$ (independently from the action of v.p. from the higher levels). Thus $\kappa_1 \kappa_2 $ has an even index on $X^l$.

 Hence, normal closure of the set $K$ coincides with $K$.
\end{proof}

\begin{statment} \label{comm}
Frattini subgroup $ \phi(G_k)= {{G_k}^{2}}\cdot [G_k,G_k]= {{G_k}^{2}} $ acts by all even permutations on ${{X}^{l}},\,\,\,0\le l<k-1$
and by all even permutations on $ X^{k}$ except for those  from \texttt{T}.
\end{statment}
\begin{proof}
Index of the automorphism $\alpha^2 $, $\alpha  \in {{S}_{\beta }}$ on $X^l$ is always even. Really the parity of the number of vertex permutations at $X^l$ in the product $({\alpha }_i {\alpha }_j)^2$, ${\alpha }_i, {\alpha }_j \in  S_{\alpha }$, $i,j<k$ is determined exceptionally by the parity of the numbers of active states of v.p. on ${{X}^{l}}$ in $\alpha $ and $\beta $ (independently of the action of v.p. from the higher levels). On $X^{k-1}$ group $G^2$ contains all automorphisms of form $\tau_{1 i},  i\leq 2^{k-1}$ which can be generated in such way $({\alpha }_{k-2} \tau_{12})^2= \tau_{1234}$, $\tau_{12}\tau_{1234} = \tau_{34}$,  $({\alpha }_{k-i} \tau_{12})^2= \tau_{1, 2, 1+2^{k-i}, 2+2^{k-i}}$ then $\tau_{1, 2, 1+2^{k-i}, 2+2^{k-i}} \tau_{12} =\tau_{1+2^{k-i}, 2+2^{k-i}}$. In such way we get set of form ${\tau_{12}, \tau_{23},, \tau_{34}, ... ,\tau_{2^{k-1}-1,2^{k-1}}}$. This set is the base for $W_{k-1}$.

The parity of the number of vertex permutations at $X^l$ in the product ${\alpha }_i$ or ${\alpha }_i {\alpha }_j$, ${\alpha }_i,{\alpha }_j \in  S_{\alpha }$) is determined exceptionally by the parity of the numbers of active v.p. on ${{X}^{l}}$ in $\alpha $ and $\beta $ (independently of the action of v.p. from the higher levels). Thus $[\alpha ,\beta ]=\alpha \beta {\alpha }^{-1}{\beta }^{-1}$ has an even number of v. p. at each level. Therefore, the commutators of the generators from ${{S}_{\alpha }}$ and elements from $G^2$ generate only the permutations with even number of v. p. at each ${{X}^{l}}$, ($0\le l\le k-2$).

Let us consider ${{\left( {{\alpha }_{0}}{{\alpha }_{l}} \right)}^{2}}={{\beta }_{l({{1,2}^{l-1}}+1)}}$. Conjugation by the element ${{\beta }_{1(1,2)}}$ (or ${{\beta }_{i(1,2)}},\,\,0<i<l$)  give us ability to express arbitrary coordinate $x:\,\,1\le x \leq {{2}^{l-1}}$ where $x=2^{k-1}-i$,
 i.e. from the element ${{\beta }_{l({{1,2}^{l-1}}+1)}}$ we can express ${{\beta }_{l(x{{,2}^{l-1}}+1)}}$. For instance $x={{2}^{j-1}}+1$, $j<l$: ${{\beta }_{l-j(1,2)}}{{\beta }_{l({{1,2}^{l-1}}+1)}}{{\beta }_{l-j(1,2)}}={{\beta }_{l({{2}^{j-1}+1},{{2}^{l-1}}+1)}}$. If $x={{2}^{l-j}}+2$ than to  realize every shift on $x$ on set $X^l$ 
the element
${{\beta }_{l({{1,2}^{l-1}}+1)}}$ should to be conjugated by such elements ${{\beta }_{l-j(1,2)}} {{\beta }_{l-1(1,2)}}$. So in such way can be realized every ${{\beta }_{l(x{{,2}^{l-1}}+1)}}$ and analogously every  ${{\beta }_{l({{2}^{l-1},y})}}$ and ${{\beta }_{l(x,y)}}$. Hence we can express from elements of $G^2$ every even number of active states of v.p. on $X^l$.
\end{proof}

Define the subgroup $G(l)<Aut{{X}^{[k]}}$, where $l< k$, as $Stab_{Aut{{X}^{[k]}}}(l)\left| _{{{X}^{l}}} \right.$. It is plain, that $G(l) \simeq  ^{Stab_{Aut{{X}^{[k]}}}(l)} /_{ Stab_{Aut{{X}^{[k]}}}(l+1)}$ because $Stab_{Aut{{X}^{[k]}}}(l+1)$ is normal subgroup of finite index in $Aut{{X}^{[k]}}$ \cite{Ne}.
Let us construct a homomorphism from $G(l)$ onto ${{C}_{2}}$ in the following way: $\varphi (\alpha )=\sum\limits_{i=1}^{{{2}^{l}}}{{{s}_{li}}(\alpha )}\bmod 2$. Note that $\varphi (\alpha \cdot \beta )=\varphi (\alpha )\circ \varphi (\beta )=(\sum\limits_{i=1}^{{{2}^{l}}}{{{s}_{li}}(\alpha )}+\sum\limits_{i=1}^{{{2}^{l}}}{{{s}_{li}}(\beta )})mod2$.

Structure of subgroup $G_{k}^{2}{{G}_{k}}'\triangleleft \underset{1}{\overset{k}{\mathop{\wr }}}\,{{S}_{2}}\simeq Aut{{X}^{[k]}}$ can be described in next way. This subgroup contains the commutant ${G}_{k}'$. So it has on each ${{X}^{l}},\,\,\,0\le l<k-1$ all even indexes that can exists there. On $X^{k-1}$ it does not exist v.p. of type \texttt{T}, which has the distance $2k-2$, rest of even the indexes are present on ${X}^{k-1}$. It's so, because the sets of elements of types \texttt{T} and \texttt{C} are not closed under operation of calculating the even power as it proved in Lemma \ref{About not closed set of element of type T}.
Thus, the squares of the elements don't belong to \texttt{T} and \texttt{C} (because they have the distance, which is less than $2k-2$).
This implies the following corollary.
\begin{corollary} \label{qoutient} A quotient group ${}^{{{G}_{k}}}/{}_{G_{k}^{2}{{G}^{'}_{k}}}$ is isomorphic to $\underbrace{{{C}_{2}}\times {{C}_{2}}\times ...\times {{C}_{2}}}_{k}$.
\end{corollary}
\begin{proof}
The proof is based on two facts $G_{k}^{2}{{G}^{'}_{k}}\simeq G_{k}^{2}\triangleleft {{G}_{k}}$ and $\left| G:G_{k}^{2}{{G}^{'}_{k}} \right|=2^k$.
Construct a homomorphism from $G_k(l)$ onto ${{C}_{2}}$ in the following way: $\varphi (\alpha )=\sum\limits_{i=1}^{{{2}^{l}}}{{{s}_{li}}(\alpha )}\bmod 2$. Note that $\varphi (\alpha \cdot \beta )=\varphi (\alpha )\circ \varphi (\beta )=(\sum\limits_{i=1}^{{{2}^{l}}}{{{s}_{li}}(\alpha )}+\sum\limits_{i=1}^{{{2}^{l}}}{{{s}_{li}}(\beta )})mod2$, where $\alpha ,\,\,\beta \in Aut{{X}^{[n]}}$.
Index of $\alpha \in G_{k}^{2}$ on ${{X}^{l}},\,l<k-1$ is even but index of  $\beta \in {{G}_{k}}$ on ${{X}^{l}}$ can be both even and odd. Note that ${{G}_{k}}(l)$ is abelian group and $G_{k}^{2}(l)\trianglelefteq {{G}_{k}}$.
  Since words with equal logarithms to all bases \cite{K} belong to distinct cosets of the commutator, the subgroup $G_{k}^{2}(l)$ is the kernel of this mapping. Also we can use homomorphism $\varphi $  which is described above and denote it as ${{\varphi }_{l}}$, to map ${{G}_{k}}(l)$ onto ${{C}_{2}}$ the $\ker {{\varphi }_{l}}=G_{k}^{2}(l)$. Really if $\alpha $ from ${{G}_{k}}(l)$ has odd number of active states of v.p. on ${{X}^{l}},\,\,\,l<k-1$ than  ${{\varphi }_{l}}(\alpha )=1$ in ${{C}_{2}}$ otherwise if this number is even than $\alpha $ from $\ker {{\varphi }_{i}}$ so ${{\varphi }_{l}}(\alpha )=0$ hence $\ker {{\varphi }_{l}}=G_{k}^{2}(l)$.
   So ${}^{{{G}_{k}}(l)}/{}_{G_{k}^{2}(l)}={{C}_{2}}$ analogously ${}^{{{B}_{k}}(l)}/{}_{B_{k}^{2}(l)}={{C}_{2}}$. Let us  check that mapping  $({{\varphi }_{0}},{{\varphi }_{1}},...,{{\varphi }_{k-2}},{{\phi }_{k-1}})$ is the homomorphism from ${{G}_{k}}$ to $\underbrace{{{C}_{2}}\times {{C}_{2}}\times ...\times {{C}_{2}}}_{k}$.
 By virtue of the fact that we can construct homomorphism ${{\varphi }_{i}}$ from every factor ${G}_{k}(i)$ of this direct product to ${{C}_{2}}$.    The group ${}^{{{G}_{k}}}/{}_{G_{k}^{2}}$  is elementary abelian 2-group because ${{g}^{2}}=e,\,\,g\in G$.

Parity of index of $\alpha \cdot \beta$ on $X^l$ is equal to sum by $\bmod 2$ of indexes of $\alpha$ and $\beta$ hence
$\varphi_l (\alpha \cdot \beta )=\left( \varphi_l (\alpha )+\varphi_l (\beta ) \right)$ because multiplication $\alpha \cdot \beta $ in ${{G}_{k}}$ does not change a parity of index of $\beta $, $\beta \in {{G}_{k}}$ on ${{X}^{l}}$.

 Really action of element of active group $A=\underbrace{{{C}_{2}}\wr {{C}_{2}}\wr ...\wr {{C}_{2}}}_{l-1}$  from wreath power $(\underbrace{{{C}_{2}}\wr {{C}_{2}}\wr ...\wr {{C}_{2}}}_{l-1})\wr {{C}_{2}}$ on element from passive subgroup ${{C}_{2}}$ of second multiplier from product $gf,\,\,g,f\in (\underbrace{{{C}_{2}}\wr {{C}_{2}}\wr ...\wr {{C}_{2}}}_{l-1})\wr {{C}_{2}}$ does not change a parity of index of $\beta $ on ${{X}^{l}}$, if index of $\beta $ was even then under action  it stands to be even and the sum $\varphi (\alpha )\bmod 2+\varphi (\beta )\bmod 2$ will be equal to $(\varphi (\alpha )+\varphi (\beta ))\bmod 2$,  hence it does not change a $\varphi (\beta )$. Since words with equal logarithms to all bases \cite{K} belong to distinct cosets of the commutator, the subgroup $G_{k}^{2}(l)$ is the kernel of this mapping.
Let us define the permutations of the type 2 that act on ${{X}_{1}}$ and ${{X}_{2}}$, where ${{X}_{1}}=\{{{v}_{k,1}},...,{{v}_{k{{,2}^{k-1}}}}\},\,\,{{X}_{2}}=\{{{v}_{k{{,2}^{k-1}}+1}},...,{{v}_{k{{,2}^{k}}}}\},\,{{X}_{1}}\cup {{X}_{2}}={{X}^{k}}$  only by even permutations. Subgroup $G_{k}^{2}G{{'}_{k}}$ acts only by permutations of type 2 on ${{X}_{1}}$, ${{X}_{2}}$, according to Statement \ref{comm}.

The restriction ${{\left. G_{k}^{2} \right|}_{{{X}^{[k-1]}}}}$ acts only by permutations of the second type (elements of it form a normal subgroup in ${{G}_{k}}$) by parity of permutation on sets  ${{X}_{1}}$ and ${{X}_{2}}$. A permutation of Type 1, where on ${{X}_{1}}$ and ${{X}_{2}}$  the group ${{G}_{k}}$  can acts by odd as well as by even permutations but in such way to resulting permutation on ${{X}^{k}}$ is always even. The number of active states from subgroup ${{G}_{k}}(k-1)$ on ${{X}^{k-1}}$ can be even as well as odd.
It means that on set of vertices of ${{X}^{k-1}}$  over ${{X}_{1}}$ i.e. vertices that are connected by edges with vertices of ${{X}^{k-1}}$  over ${{X}_{1}}$ automorphism of ${{G}_{k}}$ can contains odd number of active states (and ${{X}_{2}}$ analogously).
Hence for a subgroup $G(k-1)\simeq W_{k-1}$ such that has the normal subgroup ${{G}_{k}}^2({k-1})\triangleleft {{G}_{k}}({k-1})$ we construct a homomorphism: $ {{\phi }_{k-1}} \left( {{G}_{k}}({{X}_{k}}) \right)\to {{C}_{2}} \simeq {}^{{G}_{k}}({k-1})/{}_{ G^2_k (k-1)} $ as product of sum by $mod2$ of active states (${{s}_{k-1,i}}\in \{0,1\}$, $0<j\le {{2}^{k-2}}$ if ${{v}_{ij}}\in {{X}_{1}}$ and $X_2$ corespondently)
   on each set ${{X}_{1}}$ and ${{X}_{2}}$:  $\phi_{\alpha} ({{X}_{1}})\,\, \cdot \,\,\phi_{\alpha} ({{X}_{2}})=\sum\limits_{i=1}^{{{2}^{k-2}}}{{{s}_{k-1,i}}(\alpha)}(\bmod 2) \cdot \sum\limits_{i={{2}^{k-2}}+1}^{{{2}^{k-1}}}{{{s}_{k-1,i}}(\alpha)} (\bmod 2)$.
  Where ${{s}_{k-1,i}}(\alpha)=1$ if  there is active state  in ${{v}_{k-1,i}}, \, i<2^{k-1}+1$ and ${{s}_{k-1,i}}(\alpha)=0$ if there is no active state. It follows from structure of ${{G}_{k}}$ that $\phi_{\alpha} ({{X}_{1}})\,\,=\,\,\phi_{\alpha} ({{X}_{2}})$ so it is 0 or 1. But $G_{k}^{2}{{G}^{'}_{k}}$ admits only permutations of Type 2 on ${{X}^{k}}$ so $G_{k}^{2}{{G}_{k}}'({{X}_{k}})\triangleleft {{G}_{k}}({{X}_{k}})$ because it holds a conjugacy and it is a kernel of mapping from ${{G}_{k}}({{X}_{k}})$ onto ${{C}_{2}}$.

Hence for a subgroup $W_{k-1}$ such that has the normal subgroup ${{G}_{k}}^{2}(k-1)\triangleleft W_{k-1}$ it was constructed a homomorphism:  ${{\phi }_{k-1}}\left( W_{k-1} \right)\to {{C}_{2}}\simeq {{G}_{k}}(k-1){{/}_{G_{k}^{2}(k-1)}}$ as product of sum by $mod2$ of active states from $X_1$ and $X_2$.
 As the result we have ${}^{{{G}_{k}}}/{}_{G_{k}^{2}}\simeq \underbrace{{{C}_{2}}\times {{C}_{2}}\times ...\times {{C}_{2}}}_{k}$.
\end{proof}
Considering that it was proved in Theorem 1 and Theorem 2 that ${{G}_{k}}\simeq A_{2^k}$ we can formulate next Corollary.
   \begin{corollary} The group $Syl_2 A_{2^k}$ has a minimal generating set with $k$ generators.
\end{corollary}
   \begin{proof}
Since quotient group of ${{G}_{k}}$ by subgroup of Frattini $G_{k}^{2}{{G}^{'}_{k}}$  has minimal set of generators from $k$ elements because ${}^{{{G}_{k}}}/{}_{G_{k}^{2}{{G}^{'}_{k}}}$  is isomorphic to linear $p$-space $(p=2)$ of dimension $k$ (or elementary abelian group). Then according to theorems from \cite{Rot} $rk ({{G}_{k}})=k$. It means that $A_{2^k}$ is a group with fixed size of minimal generating set. 
   \end{proof}
 \begin{main_theorem}
The set $S_{\mathop{\beta}}=\{\mathop{\beta}_{0}, \mathop{\beta}_{1}, \mathop{\beta}_{2}, \ldots , \mathop{\beta}_{k-2}, \tau \}$, where $\mathop{\beta}_{i} = \alpha_i$, 
 is a minimal generating set for a group $G_k$ that is isomorphic to Sylow 2-subgroup of $A_{2^{k}}$.
\end{main_theorem}
We have isomorphism of $G_k$ and $Syl_2 (A_{2^k})$ from Theorem \ref{isomor}, the minimality of $S_{\mathop{\beta}}$ following from
Theorem \ref{Th about general relation} which said that $S_{\mathop{\beta}}$ has to contain an element of type \texttt{T}, Theorem 3 and Lemma \ref{rk} about minimal rank.
Another way to prove the minimality of $S_{\mathop{\beta}}$ is given to us by Corollary \ref{qoutient} about quotient by Frattini subgroup.

For example a minimal set of generators for $Syl_2(A_{8})$ can be constructed by following way, for convenience let us consider the next set:

\begin{figure}[h]
\begin{minipage}[h]{1.0\linewidth}
\center{\includegraphics[width=0.99\linewidth]{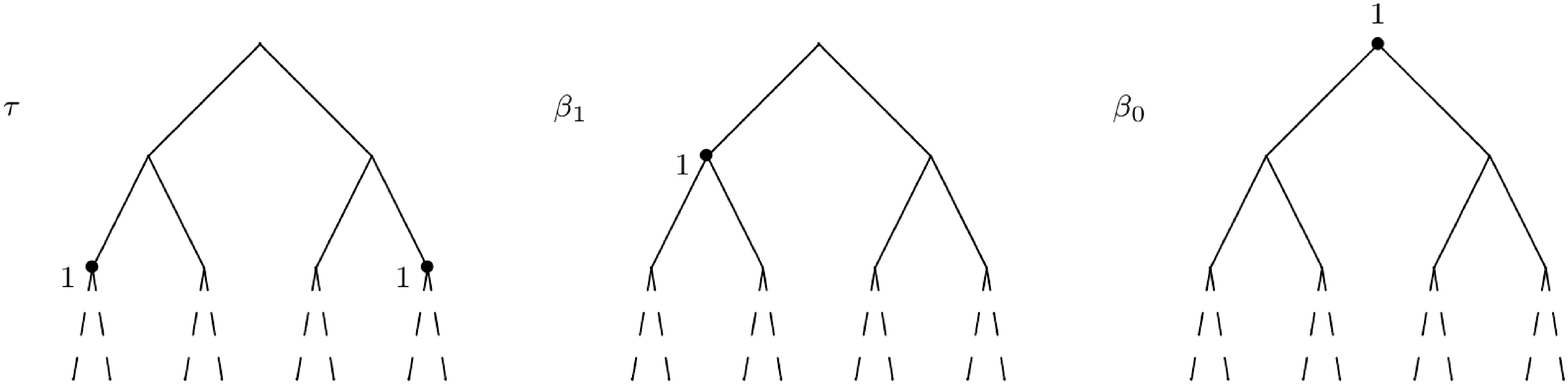} \\ Picture 1.}
\end{minipage}
\end{figure}

  Consequently, in such way we construct second generating set for $A_{2^k}$ of $k$ elements that is less than in \cite{Iv}, and this set is minimal.

We will call \emph{\textbf{diagonal base}} (${{S}_{d}}$) for $Syl_2 S_{2^k} \simeq Aut X^{[k]}$ such generating set that has the property $s_{jx}(\alpha_i)=0$ iff $ i \neq j$, (for $1\leq x\leq 2^j$) and every $\alpha_i, \, i<k$ has odd number of active v.p. A number of no trivial v.p. that can be on ${{X}^{j}}$ is odd the number of ways to chose tuple of no trivial v.p. on ${{X}^{j}}$ for generator from ${{S}_{d}}$ and equal to $2^ {2^{j}}:2=2^ {{2}^{j-1}}$.
Thus, general cardinality of ${{S}_{d}}$ for $Sy{{l}_{2}}{{S}_{{{2}^{k}}}}$ is ${{2}^{{{2}^{k}}-k-1}}$.	
There is minimum one generator of type \texttt{T} in ${{S}_{d}}$ for $Sy{{l}_{2}}{{A}_{{{2}^{k}}}}$. This generator can be chosen not less than in  $C_{{{2}^{k-2}}}^{1}C^1_{{2}^{k-2}}= {({{2}^{k-2}})}^{2} = 2^{2k-4}$ ways.  Thus, total cardinality of ${{S}_{d}}$ for $Sy{{l}_{2}}{{A}_{{{2}^{k}}}}$ is ${{2}^{{{2}^{k-1}}-k-2}}{{({{2}^{k-2}})}^{2}}$.

  And there are $k$ generators in a minimal set of generators, therefore $\left|\phi(G_k)\right| = \left| {{G}_{k}} \right|:{{2}^{k}}$ should be raised to the power of $k$. It equals to ${{(\left| {{G}_{k}} \right|:{{2}^{k}})}^{k}}={{({{2}^{{{2}^{k}}-2}}:{{2}^{k}})}^{k}}={{2}^{k({{2}^{k}}-k-2)}}$. As a result, we have ${{2}^{k({{2}^{k}}-k-1)}} \cdot ({{2}^{k}}-1)({{2}^{k}}-2)({{2}^{k}}-{{2}^{2}})...({{2}^{k}}-{{2}^{k-1}})$.

Let us consider an examples of $Syl_2 A_n$ for a cases $n=4k+r$, where $r\leq 3$.
The structure of $Syl_2A_{12}$ is the same as of the subgroup $H_{12} < Syl_2(S_8) \times Syl_2(S_4)$, for that $[Syl_2(S_8) \times Syl_2(S_4):H_{12}]=2$, $|Syl_2(A_{12})|= 2^{[12/2] + [12/4]+ [12/8]-1} = 2^9$. Also $|Syl_2(S_8)|=2^7$, $|Syl_2(S_4)|=2^3$, so $|Syl_2(S_8) \times Syl_2(S_4)|=2^{10}$ and $|H_{12}|=2^9$, because its index in $Syl_2(S_8) \times Syl_2(S_4)$ is 2. The structure of $Syl_2(A_6)$ is the same as of $H_6 < Syl_2(S_4) \times (C_2)$. Here $H_6 = \{(g,h_g)|g \in Syl_2(S_4), h_g \in C_2\}$, where
\begin{equation}\label{H}
 \begin{cases}
h_g = e, \ \ if \ g|_{X_2} \in Syl_2(A_6), \\
h_g = (5,6), \ if \, g|_{X^2} \in {Syl_2(S_6) \setminus Syl_2 A_6}.
 \end{cases}
\end{equation}
The structure of $Syl_2(A_{6})$ is the same as subgroup $H_6:$ $H_6 < Syl_2(S_4) \times (C_2)$ where $H_6= \{ (g, h) | g\in Syl_2(S_4), h \in  AutX \}$. So last bijection determined by (\ref{H}) giving us $Syl_2 A_{6} \simeq Syl_2 S_{4} $. As a corollary we have $Syl_2 A_{{2^k}+2} \simeq Syl_2 S_{2^k} $.
The structure of $Syl_2(A_{7})$ is the same as of the subgroup $H_7:$ $H_7 < Syl_2(S_4) \times S_2$ where $H_6= \{ (g, h) | g\in Syl_2(S_4), h \in  S_2 \}$ and $h$ depends of $g$:
\begin{equation}\label{HH}
 \begin{cases}
h_g = e, \ \ if \  g|_{X^2}\in  Syl_2  A_7, \\
h_g = (i,j), i,j \in\{ 5,6,7 \},  \ if \, g|_{X^2}\in  {Syl_2 S_7\setminus Syl_2A_7}.
 \end{cases}
\end{equation}
The generators of the group $H_7$ have the form $(g,h), \, \, g\in Syl_2(S_4), \, h\in C_2$, namely: $ \{ {\beta_{0}; \beta_{1}, \tau} \} \cup \{ (5,6) \}$. An element $h_g$ can't be a product of two transpositions of the set: ${(i,j), (j,k), (i,k)}$, where $i,j,k$  $\in\{ 5,6,7 \} $, because $(i,j)(j,k)=(i,k,j)$ but $ord(i,k,j) =3$, so such element doesn't belong to 2-subgroup. In general elements of $Syl_2 A_{4k+3}$ have the structure (\ref{HH}), where $h_g = (i,j), \,\, i,j \in\{ 4k+1, 4k+2, 4k+3 \}$ and $g\in Syl_2 S_{4k}$.

Also $|Syl_2(S_8)|=2^7$, $|Syl_2(S_4)|=2^3$, so $|Syl_2(S_8) \times Syl_2(S_4)|=2^{10}$ and $|H_{12}|=2^9$, because its index in $Syl_2(S_8) \times Syl_2(S_4)$ is 2. The structure of $Syl_2(A_6)$ is the same as of $H_6 < Syl_2(S_4) \times (C_2)$. Here $H_6 = \{(g,h_g)|g \in Syl_2(S_4), h_g \in C_2\}$.

The sizes of this groups are equal, really $|Syl_2(A_7)|= 2^{[7/2] + [7/4]-1}  = 2^3= |H_7|$. In case \, $g|_{L_2}\in  {S_7\setminus A_7}$ we have $C_3^2$ ways to construct one transposition that is direct factor in $H$ which complete $Syl_2 S_4$ to $H_7$ by one transposition  $: \{(5,6); (6,7); (5,7) \}$.

The structure of $Syl_2(A_{2^k+2^l})$ $(k>l)$ is the same as of the subgroup $H_{2^k+2^l} < Syl_2(S_{2^k}) \times Syl_2(S_{2^l})$, for that $[Syl_2(S_{2^k}) \times Syl_2(S_{2^l}):H]=2$. $|Syl_2(A_{2^k+2^l})|= 2^{[(2^k+2^l)!/2] + [(2^k+2^l)!/4]+ .... -1} $.
Here $H = \{(g,h_g)|g \in Syl_2(S_2^k), h_g \in Syl_2(S_2^l\}$, where

\begin{equation}\label{HHH}
 \begin{cases}
h  \in    A_{2^l}, \ \ if \  g|_{X^{k-1}}\in    A_{2^k}, \\
h:  h|_{X^2}\in  {Syl_2 (S_{2^l}) \setminus Syl_2 A_{2^l}},  \ if \, g|_{X^k}\in  {Syl_2 S_{2^k} \setminus Syl_2 A_{2^k}}.
 \end{cases}
\end{equation}
The generators of the group $H_7$ have the form $(g,h), \, \, g\in Syl_2(S_4), \, h\in C_2$, namely: ${\beta_{0}; \beta_{1}, \tau} \cup {(5,6)}$.

I.e. for element  ${{\beta }_{\sigma }}(2i-1)=2\sigma (i)-1,\,\,{{\beta }_{\sigma }}(2i)=2\sigma (i)$, ${{\sigma }_{i}}\in \left\{ {{1,2,...,2}^{k-1}} \right\}$.

Let us present new operation $\boxtimes $ (similar to that is in \cite{Dm}) as a even subdirect product of $Syl{{S}_{{{2}^{i}}}}$, $n = {{2}^{{{k}_{0}}}}+{{2}^{{{k}_{1}}}}+...+{{2}^{{{k}_{m}}}}$, $0\le {{k}_{0}}<{{k}_{1}}<...<{{k}_{m}}$ and  $m\ge 0$, i.e. $Syl{{S}_{{{2}^{{{k}_{1}}}}}}\boxtimes Syl{{S}_{{{2}^{{{k}_{2}}}}}}\boxtimes ...\boxtimes Syl{{S}_{{{2}^{{{k}_{l}}}}}}=Par(Syl{{S}_{{{2}^{{{k}_{1}}}}}}\times Syl{{S}_{{{2}^{{{k}_{2}}}}}}\times ...\times Syl{{S}_{{{2}^{{{k}_{l}}}}}})$, where $Par(G)$ -- set of all even permutations of $G$. Note, that $\boxtimes $ is not associated operation, for instance $ord({{G}_{1}}\boxtimes {{G}_{2}}\boxtimes {{G}_{3}})\,\,\,=\left| {{G}_{1}}\times {{G}_{2}}\times {{G}_{3}} \right|:2$ but $ord(({{G}_{1}}\boxtimes {{G}_{2}})\boxtimes {{G}_{3}})\,\,\,=\left| {{G}_{1}}\times {{G}_{2}}\times {{G}_{3}} \right|:4$. For cases $n=4k+1$, $n=4k+3$ it follows from formula of Legendre.

It is well known that the $Sy{{l}_{2}}{{S}_{{{2}^{{{k}_{i}}}}}}\simeq \wr _{j=1}^{{{k}_{i}}}{{C}_{2}}$. Since Sylow $p$-subgroup of direct product is direct product of Sylow $p$-subgroups and fact that  automorphism of rooted tree keeps an vertex-edge incidence relation then we have $Aut{{X}^{[{{k}_{0}}]}}\times Aut{{X}^{[{{k}_{1}}]}}\times ...\times Aut{{X}^{[{{k}_{m}}]}}\simeq Sy{{l}_{2}}{{S}_{n}}$, $n={{2}^{{{k}_{1}}}}+{{2}^{{{k}_{2}}}}+...+{{2}^{{{k}_{l}}}}$, ${{k}_{i}}\ge 0$, ${{k}_{i}}<{{k}_{i-1}}$.
Let us denote a subgroup, that consists of all even substitutions from $Sy{{l}_{2}}{{S}_{n}}$  as
$Aut{{X}^{[{{k}_{0}}]}}\boxtimes Aut{{X}^{[{{k}_{1}}]}}\boxtimes ...\boxtimes Aut{{X}^{[{{k}_{m}}]}}$, where a states of v.p. on ${{X}^{{{k}_{0}}-1}}\sqcup {{X}^{{{k}_{1}}-1}}\sqcup ...\sqcup {{X}^{{{k}_{m}}-1}}$ are related by congruence:

\begin{equation}\label{congruen}
\sum\limits_{i=0}^{m}{\sum\limits_{j=1}^{{{2}^{{{k}_{i}}-1}}}{{{s}_{{{k}_{i}}-1,j}}({{\alpha }_{i}})\equiv 0\left( \bmod 2 \right)}}.
\end{equation}

\begin{lemma} If number of active states on a last level of $Aut{{X}^{[{{k}_{i}}]}}$ from $Aut{{X}^{[{{k}_{0}}]}}\boxtimes \,\,...\,\,\boxtimes Aut{{X}^{[{{k}_{m}}]}}$ is odd, then it is subdirect product of groups $Aut{{X}^{[{{k}_{0}}]}},\,\,...\,\,,Aut{{X}^{[{{k}_{m}}]}}$.
\end{lemma}

\begin{proof}
It is a quotient group which is a homomorphic  image obtained by a mapping   from   $Aut{{X}^{[{{k}_{0}}]}}\times Aut{{X}^{[{{k}_{1}}]}}\times ...\times Aut{{X}^{[{{k}_{m}}]}} \simeq Sy{{l}_{2}}{{S}_{n}}$ to this quotient group.
A kernel of $\varphi $  consists of all automorphisms which satisfy a congruence $\sum\limits_{i=0}^{m}{\sum\limits_{j=1}^{{{2}^{{{k}_{i}}-1}}}{{{s}_{{{k}_{i}}-1,j}}({{\alpha }_{i}})\equiv 1\left( \bmod 2 \right)}}$.
 At once from definition follows, that if number of states on last level of $Aut{{X}^{[{{k}_{i}}]}}$ is odd, then a subgroup from the condition is subdirect product of groups $Aut{{X}^{[{{k}_{0}}]}},\,\,...\,\,,Aut{{X}^{[{{k}_{m}}]}}$. Actually for every state of automorphism $\alpha $ from $Aut{{X}^{[{{k}_{i}}]}}$   on ${{X}^{l}}$, $l<{{k}_{i}}-1$ we have that $(e,...,e,{{\alpha }_{i}},e,...,e)$ belongs to $Aut{{X}^{[{{k}_{0}}]}}\times Aut{{X}^{[{{k}_{1}}]}}\times ...\times Aut{{X}^{[{{k}_{m}}]}}$. An arbitrary state from ${{X}^{{{k}_{i}}-1}}$ is included in $Aut{{X}^{[{{k}_{0}}]}}\boxtimes Aut{{X}^{[{{k}_{1}}]}}\boxtimes ...\boxtimes Aut{{X}^{[{{k}_{m}}]}}$  together with even number of states from last levels of ${{X}^{[{{k}_{0}}]}},\,...\,,{{X}^{[{{k}_{m}}]}}$.  Analogous fact was proved in [1] for a direct sum of permutations groups and for their subgroups which consists of all even permutations. Our statement is a restiction on a $Sy{{l}_{2}}{{S}_{n}}$.
\end{proof}
The Sylow subgroup $Sy{{l}_{2}}({{A}_{n}})$ has index 2 in $Syl_{2}({{S}_{n}})$ and it's structure: $Syl_2{S_{2^{{{k}_{1}}}}}\boxtimes Syl_2{{S}_{{{2}^{{{k}_{2}}}}}}\boxtimes ...\boxtimes Syl_2{S}_{{2}^{{k}_{l}}}$.

\begin{lemma} \label{isomorph} If $n=4k+2$, then the subgroup $Syl_2A_n$ is isomorphic
to $Syl_2 S_{4k}$, where $k\in \mathbb{N}$.
\end{lemma}

\begin{proof}
Let us consider the subgroup
 $H_{4k+2} = \{(g,h_g)|g \in Syl_2(S_{4k}), h_g \in S_2\}$, where
\begin{equation}\label{HHHH}
 \begin{cases}
h_g = e, \ \ if \ g|_{X^k} \in Syl_2(A_{4k+2}), \\
h_g = (4k+1,4k+2), \ if \, g|_{X^k} \in Syl_2(S_{4k+2}) \setminus Syl_2(A_{4k+2}).
 \end{cases}
\end{equation}
For instance the structure of $Syl_2(A_{6})$ is the same as subgroup $H_6:$ $H_6 < Syl_2(S_4) \times (C_2)$, where $H_6= \{ (g, h) | g\in Syl_2(S_4), h \in  AutX \}$. So last bijection determined by (\ref{HHHH}) give us $Syl_2 A_{6} \simeq Syl_2 S_{4} $. As a corollary we have $Syl_2 A_{{2^k}+2} \simeq Syl_2 S_{2^k} $.

Bijection correspondence between set of elements of $Syl_2(A_n)$ and $Syl_2(S_{4k})$ we have from (\ref{HHHH}). Let's consider a mapping $\phi: Syl_2 (S_{4k}) \rightarrow Syl_2 (A_{4k+2})$ if $\sigma \in Syl_2(S_{4k})$ then $\phi(\sigma)=\sigma \circ (4k+1, 4k+2)^{\chi(\sigma)}=(\sigma,  (4k+1, 4k+2)^{\chi(\sigma)})$, where $\chi(\sigma)$ is number of transposition in $\sigma$ by module 2.
So $\phi(\sigma) \in Syl_2(A_{4k+2})$.
 If $\phi(\sigma) \in A_{n}$ then ${\chi(\sigma)}=0$, so $\phi(\sigma) \in Syl_2(A_{n-1})$. Check that $\phi$ is homomorphism.
Assume that ${{\sigma }_{1}}\in Sy{{l}_{2}}({{S}_{4k}}\backslash {{A}_{4k}}),\,\,{{\sigma }_{2}}\in Sy{{l}_{2}}({{A}_{4k}})$, then $\phi ({{\sigma }_{1}})\phi ({{\sigma }_{2}})=({{\sigma }_{1}},{{h}^{\chi({{\sigma }_{1}})}})({{\sigma }_{2}},e)=({{\sigma }_{1}}{{\sigma }_{2}},h)={{\sigma }_{1}}{{\sigma }_{2}}\circ (4k+1,4k+2)$, where $({{\sigma }_{i}},h)={{\sigma }_{i}}\circ {{h}^{\chi({{\sigma }_{i}})}}\in Sy{{l}_{2}}({{A}_{4k+2}})$. If ${\sigma _{1}},\,\,{\sigma_{2}}\in {{S}_{{{2}^{k}}}}\backslash {{A}_{{{2}^{k}}}}$, then $\phi ({{\sigma }_{1}})\phi ({{\sigma }_{2}})=({{\sigma }_{1}},{{h}^{\chi({{\sigma }_{1}})}})({{\sigma }_{2}},{{h}^{\chi({{\sigma }_{2}})}})=({{\sigma }_{1}}{{\sigma }_{2}},\,e)=(a,\,e)$, where ${{\sigma }_{1}}{{\sigma }_{2}}=a\in {{A}_{4k+2}}$.
So it is isomorphism.
\end{proof}
Let ${{n}_{m}}={{2}^{{{k}_{0}}}}+{{2}^{{{k}_{1}}}}+...+{{2}^{{{k}_{m}}}}$, where $0\le {{k}_{0}}<{{k}_{1}}<...<{{k}_{m}}$ and  $m\ge 0$.
\begin{theorem}  If ${{n}_{m}}=4k+2$, then the minimal set of generators for $Sy{{l}_{2}}{{A}_{n_m}}$ has $\sum\limits_{i=1}^{m}{{{k}_{i}}}$ elements.
\end{theorem}
\begin{proof}
Actually, according to Lemma \ref{isomorph},  $Sy{{l}_{2}}{{A}_{4k+2}}$ is isomorphic to $Sy{{l}_{2}}{{S}_{4k}}$. In Statement 2 it was proved that $Sy{{l}_{2}}{{S}_{4k}}\simeq Sy{{l}_{2}}{{S}_{{{2}^{{{k}_{1}}}}}}\times ...\times Sy{{l}_{2}}{{S}_{{{2}^{{{k}_{m}}}}}}$, where $4k={{2}^{{{k}_{1}}}}+...+{{2}^{{{k}_{m}}}}$,

${{k}_{1}}<...<{{k}_{m}}$. On the other hand, $Sy{{l}_{2}}{{S}_{{{2}^{{{k}_{i}}}}}}\simeq Aut{{X}^{[{{k}_{i}}]}}$, so there exists the homomorphism $\varphi $ from every factor $Sy{{l}_{2}}{{S}_{{{2}^{{{k}_{i}}}}}}$ onto $C_{2}^{{{k}_{i}}}$. Such homomorphism was defined in Corollary \ref{qoutient} and in \cite{Gr}.
And what is more it is known that $Aut{{X}^{[{{k}_{i}}]}}$ has a minimal generating set of $k_i$ generators \cite{Gr}.
   Thus, there exists the homomorphism from $Aut{{X}^{[{{k}_{1}}]}}\times \,\,...\,\,\times Aut{{X}^{[{{k}_{m}}]}}$ onto $C_{2}^{{{k}_{1}}}\times \,...\,\times C_{2}^{{{k}_{m}}}$, so the rank of $Syl_2 A_{4k+2}$ is $\sum\limits_{i=1}^{m}{{{k}_{i}}}$, where $k_1=1$. 
\end{proof}
This result was confirmed by the algebraic system GAP. Actually, it was founded by GAP that the minimal generating set for $Syl_2 A_{14}$, $Syl_2 A_{14} \simeq Sy{{l}_{2}}{{S}_{12}}\simeq Sy{{l}_{2}}{{S}_{{{2}^{{{2}}}}}} \times Syl_2 S_{2^3}$, of 5 elements:
 $(11,12)(13,14), (9,11)(10,12), (7,8)(9,10), (1,5)(2,6)(3,7)(4,8), (1,3)(2,4)$.

\begin{lemma} \label{Action}
 If ${{n}_{m}}\equiv 1(\bmod 2)$, then there exists a point $n$ from tuple $M$ of ${{n}_{m}}$ points indexed by numbers  from 1 to ${{n}_{m}}$, such that $S{{t}_{Sy{{l}_{2}}{{S}_{{{n}_{m}}}}}}(n)$ is isomorphic to $Sy{{l}_{2}}{{S}_{{{n}_{m}}}}$ acting on a tuple $M$.
\end{lemma}

\begin{proof}
If ${{S}_{{{n}_{m}}}}$ acts on $M$, then one of a Sylow 2-subgroups $H<{{S}_{{{n}_{m}}}}$ is isomorphic to $Aut{{X}^{[{{k}_{0}}]}}\times Aut{{X}^{[{{k}_{1}}]}}\times ...\times Aut{{X}^{[{{k}_{m}}]}}$, that acts on tuple of ${{n}_{m}}$ points, where${{n}_{m}}={{2}^{{{k}_{0}}}}+{{2}^{{{k}_{1}}}}+...+{{2}^{{{k}_{m}}}}$,  ${{k}_{0}}<{{k}_{1}}<...<{{k}_{m}}$.  By virtue of the fact that ${{n}_{m}}\equiv 1(\bmod 2)$, then ${{k}_{0}}=0$. Thus point $n$, that is in ${{X}^{[{{k}_{0}}]}}$ has a stabilizer $S{{t}_{H}}(n)\simeq Sy{{l}_{2}}{{S}_{{{n}_{m}}}}$. It is so, because group of automorphism of such group keeps an vertex-edge incidence relation of
${{X}^{\left[ {{k}_{i}} \right]}},\,\,i\in \left\{ 0,...,m \right\}$.
Thus, action of every Sylow 2-subgroup of ${{S}_{{{n}_{m}}}}$, where ${{n}_{m}}\equiv 1(\bmod 2)$, fix one element from $\{1,2,...,{{n}_{m}}\}$. \end{proof}
According to the Sylow theorem all Sylows $p$-subgroups are conjugated so their actions are isomorphic. In particular, a Sylow  2-subgroup of $S_{2^r}$ is self-normalizing. The number of Sylow 2-subgroups of ${{S}_{{{2}^{r}}}}$ is ${{2}^{r}}!:{{2}^{e}}$  where $e=1+2+...+{{2}^{r-1}}$ \cite{Weisner}.

\begin{remark} The mentioned isomorphism may be extended to $Sy{{l}_{2}}{{A}_{4k+3}}\simeq Sy{{l}_{2}}{{A}_{4k+2}}\simeq Sy{{l}_{2}}{{S}_{4k+1}}\simeq Sy{{l}_{2}}{{S}_{4k}}$.
\end{remark}
\begin{proof}

Since in accordance with Lemma \ref{Action} an action of $Sy{{l}_{2}}{{A}_{4k+3}}$ on the set of $4k+3$ elements fixes  one point, then this group as group of action is isomorphic to $Syl_2 {{A}_{4k+2}}$. For a similar reason  $Sy{{l}_{2}}{{A}_{4k+1}}\simeq Sy{{l}_{2}}{{A}_{4k}}$. As well as it was proved in Lemma \ref{isomorph} that $Sy{{l}_{2}}{{A}_{4k+2}}\simeq Sy{{l}_{2}}{{S}_{4k}}$.
\end{proof}

 The number of generating sets for $Aut X^{k_i}$ is not less then $N_{k_i}=1\cdot 2 \cdot 2^2 \cdot...\cdot 2^{k_i-1}=2^{\frac{k_i(k_i-1)}{2}}$. Hence, the number of generating sets for $Sy{{l}_{2}}{{A}_{4k+2}}$ is not less than $ 2^{\frac{k_1(k_1-1)}{2}} \cdot . . . \cdot 2^{\frac{k_m(k_m-1)}{2}} $. Thus, it can be applied in cryptography \cite{Myasn}.

\begin{property} Relation between sizes	of the Sylows subgroup for $n=4k-2$ and $n=4k$ is given by $\left| Sy{{l}_{2}}({{A}_{4k-2}}) \right|={{2}^{i}}\left| Sy{{l}_{2}}({{A}_{4k}}) \right|$, where value $i$ depends only of power of 2 in decomposition of prime number of $k$.
\end{property}
\begin{proof}
Really $\left| {{A}_{4k-2}} \right|=\frac{(4k-2)!}{2}$, therefore $\left| {{A}_{4k}} \right|=\frac{(4k-2)!}{2}(4k-1)4k$, it means that $i$ determines only by $k$ and is not bounded.
 \end{proof}




\begin{prop} 
If $n=4k$, then index $Syl_2(A_{n+3})$ in $A_{n+3}$ is equal to $[S_{4k+1}: Syl_2 (A_{4k+1})](2k+1)(4k+3)$, index
$Syl_2(A_{n+1})$ in $A_{n+1}$ as a subgroup of index $2^{m-1}$, where $m$ is the
maximal natural number, for which $4k!$ is divisible by $2^m$.
\end {prop} 
 \begin {proof} For $Syl_2(A_{n+3})$ its cardinality equal to maximal power of 2 which divide $(4k+3)!$ this power on 1 grater then correspondent power in $(4k+1)!$  because $(4k+3)!=(4k+1)!(4k+2)(4k+3)=(4k+1)!2(2k+1)(4k+3)$ so $\mid Syl_2 A_{n+3}\mid= 2^m \cdot 2 = 2^{m+1}$.
As a result of it indexes of $A_{n+3}$ and $A_{n+1}$ are following: $  [S_{4k+1}: Syl_2 (A_{4k+1}) ] = \frac{(4k+1)!}{2^m} $ and $[S_{4k+3}: Syl_2 (A_{4k+3})] = [S_{4k+1}: Syl_2 (A_{4k+1}) ](2k+1)(4k+3) = \frac{(4k+1)!}{2^m}(2k+1)(4k+3) $.
\end {proof}

\begin{prop} 
If $n=2k$ then
$[Syl_2(A_n) : Syl_2(S_{n-1})] =2^{m-1}$, where $m$ is the maximal power of 2 in factorization of $n$.
\end{prop} 
 \begin {proof}
$|Syl_2(S_{n-1})|$ is equal to $t$ that is a maximal power of 2 in $(n-1)!$. $|Syl_2(A_{n})|$ is equal to maximal power of 2 in $(n!/2)$. Since $n=2k$ then $(n/2)!=(n-1)!\frac{n}{2}$ and $2^f$ is equal to product maximal power of 2 in $(n-1)!$ on maximal power of 2 in $\frac{n}{2}$. Therefore $\frac{|Syl_2(A_{n})|}{|Syl_2(S_{n-1})|}=\frac {2^{m-1}}{2^t} 2^t=2^{m-1}. $
Note that for odd $m=n-1$ the group $Syl_2(S_{m}) \simeq Syl_2(S_{m-1})$ i.e. $Syl_2(S_{n-1})\simeq Syl_2(S_{n-2})$. The group $Syl_2(S_{n-2})$ contains the automorphism of correspondent binary subtree with last level $X^{n-2}$ and this automorphism realizes the permutation $\sigma$ on $X^{n-2}$. For every $\sigma\in Syl_2(S_{n-2})$ let us set in correspondence a permutation $\sigma (n-1,n)^{\chi (\sigma)} \in Syl_2(A_{n})$, where $\chi (\sigma)$ -- number of transposition in $\sigma$ by $mod\, 2$, so it is bijection $\phi(\sigma)\longmapsto \sigma (n-1, n){\chi (\sigma)}$ that has property of homomorphism, see Lemma \ref{isomorph}. Thus, we prove that $Syl_2(S_{n-1}) \hookrightarrow Syl_2(A_{n})$ and its index is $2^{d-1}$.
\end {proof}

\begin{prop} 
The ratio of $|Syl_2(A_{4k+3})|$ and $|Syl_2(A_{4k+1})|$ is equal to 2 and ratio of indexes $[A_{4k+3} : Syl_2(A_{4k+3})]$ and $[A_{4k+1} : Syl_2(A_{4k+1})]$ is equal $(2k+1)(4k+3)$.
\end{prop} 

\begin{proof} The
ratio $|Syl_2(A_{4k+3})| : |Syl_2(A_{4k+1})|= 2$ holds because formula of Legendre gives us new one power of 2 in $(4k+3)!$ in compering with $(4k+1)!$.  Second part of statement follows from theorem about $p$-subgroup of $H$, $[G:H] \neq kp $ then one of $p$-subgroups of $H$ is Sylow $p$-group of $G$. In this case $p=2$ but $|Syl_2(A_{4k+3})| : |Syl_2(A_{4k+1})|=2$ so we have to divide ratio of indexes on $2$.
\end{proof}

\begin{prop} 
If $n=2k+1$ then
$Syl_2(A_n) \cong Syl_2(A_{n-1})$ and
$Syl_2(S_n) \cong Syl_2(S_{n-1})$.
\end{prop} 
 \begin {proof}
 Sizes of these subgroups are equal to each other according to Legender's formula which counts power of 2 in $(2k+1)!$ and $(2k)!$ we obtain  that these powers are equal. So these maximal 2-subgroups are isomorphic. From Statment 1 can be obtained that vertex with number $2k+1$ will be fixed to hold even number of transpositions on $X^{k_1}$ from decomposition of $n$ which is in Statement 1. For instance $Syl_2(A_{7})\simeq Syl_2(A_{6})$ and by the way $Syl_2(A_{6})\simeq C_2 \wr C_2 \simeq D_4$, $Syl_2(A_{11})\simeq Syl_2(A_{10}) \simeq C_2 \wr C_2 \wr C_2 $.
\end {proof}

 \begin {proof}
 Sizes of these subgroups are equal to each other according to Legender's formula which counts power of 2 in $(2k+1)!$ and $(2k)!$ we obtain  that these powers are equal. So these maximal 2-subgroups are isomorphic. From Statment 1 can be obtained that vertex with number $2k+1$ will be fixed to hold even number of transpositions on $X^{k_1}$ from decomposition of $n$ which is in Statement 1. For instance $Syl_2(A_{7})\simeq Syl_2(A_{6})$ and by the way $Syl_2(A_{6})\simeq C_2 \wr C_2 \simeq D_4$, $Syl_2(A_{11})\simeq Syl_2(A_{10}) \simeq C_2 \wr C_2 \wr C_2 $.
\end {proof}

Let us denote by  $S(n)$ and $S[n]$ a minimal generating system of $Sy{{l}_{2}}({{A}_{n}})$, $Sy{{l}_{2}}({{S}_{n}})$ correspondently.
Let ${{n}_{m}}={{2}^{{{k}_{0}}}}+{{2}^{{{k}_{1}}}}+...+{{2}^{{{k}_{m}}}}$, where $0\le {{k}_{0}}<{{k}_{1}}<...<{{k}_{m}}$  and  $m\ge 0$.
\begin{definition}  We shall call the top of system of generators a portion of generators from subgroup with the largest degree ${{k}_{m}}$, which belong to tuples, that represent elements from subdirect product,  where there are not other non-trivial elements.
\end{definition}

\begin {theorem} Any minimal set of generators for $Sy{{l}_{2}}{{A}_{n}}$ has $\sum\limits_{i=0}^{m}{{{k}_{i}}}-1$ generators, if $m>0$, and it has ${{k}_{0}}$ generators, if $m=0$.
 \end {theorem}
\begin{proof}
Construction of generating set $S({{n}_{m}})$ is such that it contains on the first  coordinate all generators from $Sy{{l}_{2}}{{S}_{{{2}^{{{k}_{0}}}}}}$, it contains  all generators of $Sy{{l}_{2}}{{S}_{{{2}^{{{k}_{1}}}}}}$ on the second one, analogously it contains on $i$-th coordinate  all generators of $Sy{{l}_{2}}{{S}_{{{2}^{{{k}_{i}}}}}}$,  $0<i\le m$. Hence we can generate on $i$-th coordinate an arbitrary element from  $Sy{{l}_{2}}{{S}_{{{2}^{{{k}_{i}}}}}}$.  States on ${{X}^{{{k}_{0}}-1}}\sqcup {{X}^{{{k}_{1}}-1}}\sqcup ...\sqcup {{X}^{{{k}_{m}}-1}}$ are related by congruence:
\begin{equation}\label{cong}
\sum\limits_{i=0}^{m}{\sum\limits_{j=1}^{{{2}^{{{k}_{i}}-1}}}{{{s}_{{{k}_{i}}-1j}}(\alpha_i )\equiv 0\left( \bmod 2 \right)}} 	
\end{equation}
This congruence holds due to structure of constructed by us system of generators.

Let us consider a construction of a system of generators for $Sy{{l}_{2}}{{A}_{n}}$ and prove its minimality. Consider an induction base on the example of the group ${{A}_{28}}$, since  ${{A}_{28}}=Sy{{l}_{2}}{{A}_{4}}\times Sy{{l}_{2}}{{A}_{8}}\times Sy{{l}_{2}}{{A}_{16}}$, according to the proved fact  it is equivalent to a construction of minimal system of generators for $Sy{{l}_{2}}{{A}_{28}}=Sy{{l}_{2}}{{S}_{4}}\boxtimes Sy{{l}_{2}}{{S}_{8}}\boxtimes Sy{{l}_{2}}{{S}_{16}}$. All elements of this system define odd substitutions on the respective sets ${{X}^{2}}$ and ${{X}^{3}}$  by the generators of the subgroups $Aut{{X}^{[1]}}=Sy{{l}_{2}}{{S}_{4}}=\left\langle {{\alpha }_{0,1(1)}},{{\alpha }_{1}} \right\rangle $, and $Sy{{l}_{2}}{{S}_{8}}$  where $\left\langle {{\alpha }_{0,2(1)}},{{\alpha }_{1(2),2(1)}},{{\alpha }_{2}} \right\rangle =Aut{{X}^{[3]}}=Sy{{l}_{2}}{{S}_{8}}$, ${{\alpha }_{2}}={{\alpha }_{2(1)}}$. Last subgroup of the subdirect  product has the next generating system $Sy{{l}_{2}}{{S}_{16}}=\left\langle {{\alpha }_{0}},\,\,{{\alpha }_{1(1)}},\,\,{{\alpha }_{2(1)}},\,\,{{\alpha }_{3(1)}} \right\rangle $  hence here only the last element  determines  odd permutation on ${{X}^{4}}$.
Other generators for  $Sy{{l}_{2}}{{S}_{8}}$ can be decomposed in a product of two special elements ${{\alpha }_{2}}{{\alpha }_{i}}={{\alpha }_{i(1);2(1)}}$. We call \emph{such structure} of element as \emph{odd structure}.
Generators  of  $Syl_{2}{S_4}$,  $0 \le i<2$ have the same structure ${{\alpha }_{1}}{{\alpha }_{0}}={{\alpha }_{0;1(1)}}$, where $0\le i<1$. Hence a product of two arbitrary elements from the set $\left\{ {{\alpha }_{0,2(1)}},{{\alpha }_{1(2),2(1)}},{{\alpha }_{2}} \right\}$ is automorphism that has even number of active states of v.p. on ${{X}^{2}}$. Analogously statement  is true for $\left\{ {{\alpha }_{0,1(1)}},{{\alpha }_{1}} \right\}$ on ${{X}^{1}}$ and for ${{\alpha }_{1(1)}},\,\,{{\alpha }_{2(1)}},\,\,{{\alpha }_{3(1)}}$ on ${{X}^{3}}$.  Now  rename  generators of group  $Sy{{l}_{2}}{{S}_{4}}$ as  ${{\beta }_{0,1(1)}}={{t}_{0}},\,\,{{\beta }_{1}}={{t}_{1}}$, for the second group as ${{\alpha }_{0,1(1),2(1)}}={{s}_{0}},\,\,{{\alpha }_{1,2(1)}}={{s}_{1}},\,\,{{\alpha }_{2}}={{s}_{2}}$, rename  generators for $Sy{{l}_{2}}{{S}_{16}}$ as ${{f}_{0}}={{\alpha }_{0}},\,...\,,{{f}_{3}}={{\alpha }_{3(1)}}$. A characteristic feature of the built elements is, that those of them, which are generators of the same subgroup $Sy{{l}_{2}}{{S}_{{{2}^{i}}}}$ and implement odd substitutions, contains the unique common 2-cycle, that implements odd substitution on the set, where the correspondent  group acts. Thus, for  $Aut{{X}^{[2]}}$ it is the cycle $(1,2)$ on set with 4 elements. Therefore, according to the given by us structure of generators of groups $Sy{{l}_{2}}{{S}_{{{2}^{i}}}}$, those of them, that implement odd substitutions, their products and their squares, are already even substitutions.
The system of generators for $Sy{{l}_{2}}{{A}_{28}}$, $Sy{{l}_{2}}{{A}_{28}}=Sy{{l}_{2}}{{S}_{4}}\boxtimes Sy{{l}_{2}}{{S}_{8}}\boxtimes Sy{{l}_{2}}{{S}_{16}}$ is $\left\langle({{s}_{0}},{{t}_{0}},e), (t_0, e, f_3), (e, s_{1}, f_3), ({{s}_{1}},{{t}_{0}},e),(e,{{t}_{1}},{{f}_{3}}),(e,{{t}_{2}},{{f}_{3}}),(e,e,\,{{f}_{2}}),(e,e,\,{{f}_{1}}),(e,e,\,{{f}_{0}}) \right\rangle $ that is why substitutions, that are defined by generators ${{f}_{2}},\,\,{{f}_{1}},\,\,{{f}_{0}}$ must be even (we shall call just this \textbf{\emph{set as top}}) and ${{f}_{3}}$ must has odd structure.


Since for the group $Sy{{l}_{2}}{{S}_{{{2}^{{{k}_{i}}}}}}$ there exists a surjective homomorphism $({{\varphi }_{0}},\,{{\varphi }_{1}},\,{{\varphi }_{2}},...,\,{{\varphi }_{{{k}_{i}}-1}})$ onto $C_{2}^{k_{i}}$, where ${{\varphi }_{i}}$ is homomorphism from ${{G}_{{{k}_{i}}}}(l)$ onto ${{C}_{2}}$, so there  exists a surjective homomorphism from $Sy{{l}_{2}}{{S}_{4}}\times Sy{{l}_{2}}{{S}_{8}}\times Sy{{l}_{2}}{{S}_{16}}$  onto $C_{2}^{9}$ \cite{Gr}.
We apply it to every ${{G}_{{{k}_{i}}}}(l)$, which are the subgroups of $Sy{{l}_{2}}{{S}_{4}}\times Sy{{l}_{2}}{{S}_{8}}\times Sy{{l}_{2}}{{S}_{16}}$, where ${{k}_{i}}\in \left\{ 2,\,3,\,4 \right\}$ and $l$ changes from 0 to ${{2}^{{{k}_{i}}}}-1$ for each ${{k}_{i}}$.

Show, that there is enough $\sum\limits_{i=0}^{m}{{{k}_{i}}}-1$ generators in general case, when $m>0$. Arbitrary element from   can be generated on the first coordinate as follows $({{s}_{0}},{{\beta }_{1}},e)({{s}_{1}},{{\beta }_{1}},e)=({{\alpha }_{0}},e,e)$ and $({{\alpha }_{0}},e,e)({{s}_{0}},{{\beta }_{1}},e)=({{\alpha }_{1(2)}},{{\beta }_{1}},e)$  then $({{\alpha }_{1}},{{\beta }_{1}},e)({{\alpha }_{0,1(1)}},{{\beta }_{1}},e)=({{\alpha }_{1(12)}},e,e)$. Hence, we have derived on the first coordinate all non-trivial elements from $Sy{{l}_{2}}{{A}_{4}}$.

 All automorphisms, which has  odd structure,  provides  an odd permutations  on ${{X}^{{{k}_{i}}}}$, which contain in its cyclic decomposition the unique cycle of length 2 (in  common case it is  odd number of such cycles) that rearranges neighboring elements  from $X^{{{k}_{i}}}$ and arises under the action of active state from ${{X}^{{{k}_{i}}-1}}$, all other cycles in substitutions are even, because they are defined by states from  ${{X}^{k-j}},\,\,1<j<{{k}_{i}}$. Thus, in the multiplication of these automorphisms, reduction of such cycle is obtained or one more conjugated  to it 2 cycle appears in the correspondent substitution.

The considered by us group $Sy{{l}_{2}}{{S}_{4}}\boxtimes Sy{{l}_{2}}{{S}_{8}}\boxtimes Sy{{l}_{2}}{{S}_{16}}$ is a quotient group of $Sy{{l}_{2}}{{S}_{4}}\times Sy{{l}_{2}}{{S}_{8}}\times Sy{{l}_{2}}{{S}_{16}}$ by normal \emph{closure} ${{R}^{G}}$, where $R$ is the relation
that can be derived from
 the congruence \ref{cong}. Thus from the group $Sy{{l}_{2}}{{S}_{4}}\boxtimes Sy{{l}_{2}}{{S}_{8}}\boxtimes Sy{{l}_{2}}{{S}_{16}}$  has on 1 generator less than  $Sy{{l}_{2}}{{S}_{4}}\times Sy{{l}_{2}}{{S}_{8}}\times Sy{{l}_{2}}{{S}_{16}}$ because we have new relation from which this generator can be expressed
 from form ${{t}^{-1}}\omega ,\,\,\omega \in Sy{{l}_{2}}{{A}_{28}}$ and $t$  is one of generators for $Sy{{l}_{2}}{{S}_{4}}\times Sy{{l}_{2}}{{S}_{8}}\times Sy{{l}_{2}}{{S}_{16}}$. Hence now $t$ can be express due to a new relation similar
as it was in transformation of Tietce. Thus we have on one generator less because $t$ is depends from rest of generators for $Sy{{l}_{2}}{{S}_{4}}\times Sy{{l}_{2}}{{S}_{8}}\times Sy{{l}_{2}}{{S}_{16}}$.

For the group $Sy{{l}_{2}}{{A}_{28}}$ there exists a surjective homomorphism $({{\varphi }_{0}},\,{{\varphi }_{1}},\,{{\varphi }_{2}},\,{{\varphi }_{{{k}_{3}}}})$ onto $C_{2}^{8}$, which is built previously. We apply it to every ${{G}_{{{k}_{i}}}}(l)$, which are the subgroups of $Sy{{l}_{2}}{{S}_{4}}\boxtimes Sy{{l}_{2}}{{S}_{8}}\boxtimes Sy{{l}_{2}}{{S}_{16}}$, where ${{k}_{i}}\in \left\{ 2,\,3,\,4 \right\}$ and  $l$ changes from $0$  to ${{2}^{{{k}_{i}}}}-1$ for each ${{k}_{i}}$.
The main property of a homomorphism holds for all ${{G}_{{{k}_{i}}}}(l)$.
Actually, a parity of index of $\alpha \cdot \beta $  on ${{X}^{l}}$ is equal to sum by $\bmod \,2$  of indexes of $\alpha $ and $\beta $  hence
$\varphi ({{\alpha }_{l}}\cdot {{\beta }_{l}})=\left( \varphi ({{\alpha }_{l}})+\varphi ({{\beta }_{l}}) \right)$ because multiplication  $\alpha \cdot \beta $ in  ${{G}_{k}}$ does not change a parity of index of $\beta $,  $\beta \in {{G}_{k}}$ on ${{X}^{l}}$.
This parity depends only from parities of $\alpha $ and $\beta $ on ${{X}^{l}}$. That's why the described mapping is a homomorphism onto ${{C}_{2}}$. The same is true for every $l\in \{0,...,3\}$ and  for every ${{k}_{i}}\in \left\{ 2,\,3,\,4 \right\}$ from  ${{G}_{{{k}_{i}}}}(l)$. A group ${{A}_{{{n}_{m}}}}$ has the similar structure $Sy{{l}_{2}}{{S}_{{{2}^{{{k}_{0}}}}}}\boxtimes \ldots \boxtimes Sy{{l}_{2}}{{S}_{{{2}^{{{k}_{m-1}}}}}}\boxtimes Sy{{l}_{2}}{{S}_{{{2}^{{{k}_{m}}}}}}$ and main property of constructed homomorphism $\varphi ({{\alpha }_{l}}\cdot {{\beta }_{l}})=\left( \varphi ({{\alpha }_{l}})+\varphi ({{\beta }_{l}}) \right)$ holds by the same reason for every, $l\in \left\{ 0,...,{{k}_{m}}-1 \right\}$. 

 Let us separate out a subgroup which corresponds to the last level, at which the active states exist. The parity of the index of this subgroup, let it be ${{G}_{4}}(3)$, depends also on the parities of all other level subgroups of  $Aut{{X}^{[2]}}\boxtimes Aut{{X}^{[3]}}\boxtimes Aut{{X}^{[4]}}$ because relation \ref{cong} holds.
 Thus $Sy{{l}_{2}}{{A}_{28}}$ has minimal generating system from 8 elements.

Let us present $S(28)$ in form useful for generalization and for counting rank of $S(28)$: \\ $\left\langle ({{s}_{0}},e,{{f}_{3}}),({{s}_{1}},e,{{f}_{3}}),(e,\,{{t}_{0}},\,{{f}_{3}}),(e,{{t}_{1}},{{f}_{3}}),(e,{{t}_{2}},{{f}_{3}}),(e,e,\,{{f}_{2}}),(e,e,\,{{f}_{1}}),(e,e,\,{{f}_{0}}) \right\rangle $. \\ So $rk(S(28))$=8. Obviously that generator ${{f}_{3}}$ from generating system of $Sy{{l}_{2}}{{S}_{16}}$ occurs only in combination with  generators of other subgroups and any other of generators occurs only one time. Thus
formula $\sum\limits_{i=0}^{m}{{{k}_{i}}}-1$ holds.
This result was confirmed by algebraic system GAP by which it was founded the minimal generating set for $Syl_2 A_{28}$ from 8 elements: $(25,27)(26,28),  (23,24)(25,26),  (17,21)(18,22)(19,23)(20,24),  (17,19)(18,20),\\  (15,16)(17,18),
  (1,9)(2,10)(3,11)(4,12)(5,13)(6,14)(7,15)(8,16),  (1,5)(2,6)(3,7)(4,8),\\ (1,3)(2,4)$. Thus, the base of induction is checked.

The construction of homomorphism and minimal generating set for case $n={{2}^{k}}$ was fully investigated in Corollary \ref{qoutient}.

Show,  that there is enough $\sum\limits_{i=0}^{m}{{{k}_{i}}}-1$ generators in general case, when $m>0$.
There are all ${{k}_{i}}$ generators of  $Sy{{l}_{2}}{{A}_{{{2}^{{{k}_{i}}}}}}$ on  $i$-th coordinate. Consider the example from the base of induction  arbitrary element from $Sy{{l}_{2}}{{A}_{4}}$ can be generated on the first coordinate as follows $({{s}_{0}},{{\beta }_{1}},e)({{s}_{1}},{{\beta }_{1}},e)=({{\alpha }_{0}},e,e)$ and $({{\alpha }_{0}},e,e)({{s}_{0}},{{\beta }_{1}},e)=({{\alpha }_{1(2)}},{{\beta }_{1}},e)$  then $({{\alpha }_{1}},{{\beta }_{1}},e)({{\alpha }_{0,1(1)}},{{\beta }_{1}},e)=({{\alpha }_{1(12)}},e,e)$. Hence, we have derived on the first coordinate all non-trivial elements from $Sy{{l}_{2}}{{A}_{4}}$.
 All automorphisms, which  has  odd structure,  provides  an odd permutations  on ${{X}^{{{k}_{i}}}}$, which contain in its cyclic decomposition the unique cycle of length 2 (in  common case it is  odd number of such cycles) that rearranges neighboring elements  from $X^{{{k}_{i}}}$ and arises under the action of active state from ${{X}^{{{k}_{i}}-1}}$, all other cycles in substitutions are even, because they are defined by states from  ${{X}^{k-j}},\,\,1<j<{{k}_{i}}$. Thus, in the multiplication of these automorphisms, reduction of such cycle is obtained or one more conjugated  to it 2 cycle appears in the correspondent substitution.

So, closure holds with respect to parity of obtained  tuples  of  elements as a result of multiplication in the whole subdirect product, because in the product automorphism is reduced on the second coordinate, which is determined   by the same state, that is determined  by ${{\beta }_{3}}$. It is so because it has the order 2 therefore $e$ are on other coordinates, and automorphism on the first coordinate becomes an even substitution already after the first multiplication by $({{s}_{0}},{{\beta }_{1}},e)$. Indeed, system of generators for $\left\langle {{t}_{0}},\,{{t}_{1}} \right\rangle $ has such cyclic structure, that the same cycle implements the unique odd substitution - this is the cycle of length 2, therefore it is reduced as a result of multiplication of two such generators and even substitution is obtained, the generators ${{s}_{0}},\,{{s}_{1}},\,{{s}_{2}}$,   ${{f}_{0}},{{f}_{1}},{{f}_{2}},{{f}_{3}}$  for $Sy{{l}_{2}}{{A}_{8}}$ and $Sy{{l}_{2}}{{A}_{16}}$  correspondently have the same structure, so this is the system of generators, that has only 7 elements, while the direct product of these groups should have 9 elements.  Thus, for this example the statement is true. Let us show, that minimal system of generators can be constructed  by the same method for ${A_{n_m}}$ that contains in subdirect product the factor $Sy{{l}_{2}}{{S}_{{2^{{k_m}}}}}$.

All generator except those from the top of generating  system for $Sy{{l}_{2}}{{S}_{{{2}^{{{k}_{m}}}}}}$ has odd structure. For $Sy{{l}_{2}}{{S}_{{{2}^{{{k}_{m}}}}}}$ we choose one generator  that realize odd permutation on ${{X}^{{{k}_{m}}}}$ analogously as ${{\alpha }_{{{k}_{m}}-1}}$ and other generators ${{\alpha }_{{{k}_{0}}}},...,{{\alpha }_{{{k}_{m}}-2}}$ belong to top of system of generators   of $Sy{{l}_{2}}{{S}_{{{2}^{{{k}_{m}}}}}}$, so they define even substitutions on ${{X}^{{{k}_{m}}}}$. Here every separately taken generator for the separated subgroup $Sy{{l}_{2}}{{S}_{{{2}^{{{k}_{i}}}}}},\,\,\,{{k}_{i}}<{{k}_{m}}$  defines odd substitution and has odd structure, only generators from \emph{the top of system of generators} of the subgroup $Sy{{l}_{2}}{{S}_{{{2}^{{{k}_{m}}}}}}$ implement even substitutions. Only one generator from $S({{2}^{{{k}_{m}}}})$ has odd structure and only this generator stands in pair with generator from $S({{2}^{{{k}_{m}}-1}})$ in  tuple of generators of subdirect product for $Sy{{l}_{2}}{{S}_{{{2}^{{{k}_{0}}}}}}\boxtimes Sy{{l}_{2}}{{S}_{{{2}^{{{k}_{1}}}}}}\boxtimes \,\,...\,\,\boxtimes Sy{{l}_{2}}{{S}_{{{2}^{{{k}_{m}}}}}}$. Thus  rank  of $S({{n}_{m}})$ in comparing to $S({{n}_{m-1}})$ will grows on ${{k}_{m}}$. Thus rank of $S({{n}_{m}})$ became to be $\sum\limits_{i=1}^{m}{{{k}_{i}}}-1$ as it states in formula from condition.

For the subgroup $Sy{{l}_{2}}{{S}_{{{2}^{{{k}_{i}}}}}},\,\,i<m$, distinguish its subgroups of the form ${{G}_{{{k}_{i}}}}(l),\,\,l<{{k}_{i}}-1$ elements from this subgroup act on ${{X}^{{{k}_{i}}}}$ by even substitutions and therefore are not linked with elements ${{G}_{{{k}_{j}}}}(s),\,\,s<{{k}_{j}}-1$, where $i\ne j$, ${{k}_{i}}-s={{k}_{j}}-l$ by parity relation and they forming a direct product. States from last levels ${{X}^{{{k}_{0}}}},...,{{X}^{{{k}_{m}}}}$ are linked by parity relation and form the subdirect product, in which there is only even number of states in every tuple with $m+1$  coordinates.
Construction of generating system $S({{n}_{m}})$  is such that it presents  on the first  coordinate all generators from  $Sy{{l}_{2}}{{S}_{{{2}^{{{k}_{0}}}}}}$, it presents  all generators of $Sy{{l}_{2}}{{S}_{{{2}^{{{k}_{1}}}}}}$ on the second one, analogously it presents on $i $-th coordinate  all generators of $Sy{{l}_{2}}{{S}_{{{2}^{{{k}_{i}}}}}}$, $0<i\le m$. Hence we can generate on i-th coordinate an arbitrary element from $Sy{{l}_{2}}{{S}_{{{2}^{{{k}_{i}}}}}}$. A states on last level of ${{X}^{[{{k}_{0}}-1]}}\sqcup {{X}^{[{{k}_{1}}-1]}}\sqcup ...\sqcup {{X}^{[{{k}_{m}}-1]}}$ are related by congruence \ref{cong}.
And form the subgroup of $H=Sy{{l}_{2}}{{S}_{{{2}^{{{k}_{0}}}}}}\boxtimes \,\,\,...\,\,\,\boxtimes Sy{{l}_{2}}{{S}_{{{2}^{{{k}_{m}}}}}}$ that is a quotient group of $G=Sy{{l}_{2}}{{S}_{{{2}^{{{k}_{0}}}}}}\times \,\,\,...\,\,\,\times Sy{{l}_{2}}{{S}_{{{2}^{{{k}_{m}}}}}}$  by normal c  closure ${{R}^{G}}$, where $R$ is the relation equivalent to relation \ref{cong}. Since as in proof of induction base we have that the group $H$  has on 1 generator less than $G$ because we have new relation \ref{cong}
that does not belongs to normal closure of relations from the group $G$ and $R$ does not contains new generators  than those for $G$, therefore one of old generators can be expressed.
This congruence holds due to the structure of constructed by us system of generators.

 To prove that proposed by us system generates whole group $Sy{{l}_{2}}{{A}_{n}}$ let us prove that there is equality of orders of $\left\langle S(n) \right\rangle $ and $Sy{{l}_{2}}{{A}_{n}}$.
Order of $Sy{{l}_{2}}{{S}_{n}}$ is ${{2}^{{{2}^{{{k}_{0}}}}+...+{{2}^{{{k}_{m}}}}-(m+1)}}={{2}^{n-m-1}}$ because  $Syl{{S}_{{{2}^{k}}}}\simeq Aut{{X}^{[{{2}^{k}}]}}$  and its order is equal to ${{2}^{{{2}^{k}}-1}}$, where $n={{2}^{{{k}_{0}}}}+{{2}^{{{k}_{1}}}}+...+{{2}^{{{k}_{m}}}}$. Thus  $\left| Sy{{l}_{2}}{{A}_{n}} \right|={{2}^{n-m-2}}$.
But order of constructed by us quotient group $Sy{{l}_{2}}{{S}_{{{2}^{{{k}_{0}}}}}}\boxtimes Sy{{l}_{2}}{{S}_{{{2}^{{{k}_{1}}}}}}\boxtimes \,\,...\,\,\boxtimes Sy{{l}_{2}}{{S}_{{{2}^{{{k}_{m}}}}}}$ is also equal to ${{2}^{n-m-2}}$ because the relation \ref{cong} restricts class of states from last level of ${{X}^{[{{k}_{0}}-1]}}\sqcup {{X}^{[{{k}_{1}}-1]}}\sqcup ...\sqcup {{X}^{[{{k}_{m}}-1]}}$  to the subclass that realize even permutations on ${{X}^{n}}$.  The same order for $\frac{n!}{2}$  give us Legender's formula for finding maximal power of 2 which divides $\frac{n!}{2}$:

 $\left[ \frac{n}{2} \right]+...+\left[ \frac{n}{{{2}^{{{k}_{m}}}}} \right]-1={{2}^{{{k}_{0}}}}+{{2}^{{{k}_{1}}}}+...+{{2}^{{{k}_{m}}}}-1=(n-m-1)-1=n-m-2$.
\end{proof}

Also from Statement \ref{comm} and corollary from it about $(AutX^{[k]})'$ can be deduced that derived length of $Syl_2 A_2^k$ is not always equal to $k$ as it was said in Lemma 3 of \cite{Dm} because in case $A_{2^k}$ if $k=2$ its $Syl_2 A_4 \simeq K_4$ but $K_4$ is abelian group so its derived length is 1.




\section{Some applications of constructed generating systems}

 As it was calculated in previous paragraph a total cardinality of ${{S}_{d}}$ for $Sy{{l}_{2}}{{A}_{{{2}^{k}}}}$ is ${{2}^{{{2}^{k-1}}-k-2}}{{({{2}^{k-2}})}^{2}}$.
Thus, if we associate generating set with alphabet and choice of generating set will be a private key, then it can be applied in cryptography \cite{Myasn}.
  More over our group $G_k$ has exponential grows of different generating sets and diagonal bases that can be used for extention of key space. Diagonal bases are useful for easy constructing of normal form \cite{Ushak} of an element $g\in {{G}_{k}}$.

Let us consider a function of Morse \cite{Shar} $f:\,{D^2}\to \mathbb{R}$ that painted at pict. 2 and graph of Kronrod-Reeb \cite{Maks} that obtained by contraction every set's component  of level of ${{f}^{-1}}(c)$ in point. Group of automorphism of this graph is isomorphic to $Sy{{l}_{2}}{{S}_{{{2}^{k}}}}$, where $k=2$ in general case we have regular binary rooted tree for arbitrary $k\in \mathbb{N}$.


\begin{figure}[h]
\begin{minipage}[h]{0.49\linewidth}
\center{\includegraphics[width=0.9\linewidth]{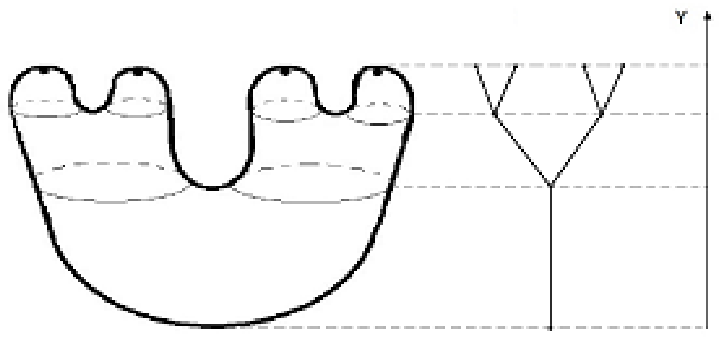} \\ Picture 2.}
\end{minipage}
\end{figure}
According to investigations of \cite{Maks2} for $D^2$ we have that $Syl_{2}S_{2^k} > G_k \simeq Syl_{2} A_{2^k}$ is quotient group of diffeomorphism group that stabilize function and isotopic to identity. Analogously to investigations of \cite{Maks, Maks2, SkThes} there is short exact sequence
$0\to {{\mathbb{Z}}^{m}}\to {{\pi }_{1}}{{O}_{f}}(f)\to G\to 0$, where $G$-group of automorphisms Reeb's (Kronrod-Reeb) graph \cite{Maks} and $O_f(f)$ is orbit under action of diffeomorphism group, so it could be way to transfer it for a group $Sy{{l}_{2}}(S_{2^k})$, where $m$ in ${{\mathbb{Z}}^{m}}$ is number of inner edges or vertices in Reeb's graph, in case for $Syl_{2}S_{4}$ we have $m=3$.

Higher half of projection of manifold from pic. 2 can be determed by product of the quadratic forms $-({{(x+4)}^{2}}+{{y}^{2}})({{(x+3)}^{2}}+{{y}^{2}})({{(x-3)}^{2}}+{{y}^{2}})({{(x-4)}^{2}}+{{y}^{2}})=z$ in points $(-4,0) (-3,0) (3,0) (4,0) $ it reach a maximum value 0. Generally there is $-d_{1}^{2}d_{2}^{2}d_{3}^{2}d_{4}^{2}=z$.

\section{ Conclusion }
The proof of minimality of constructed generating sets was done, also the description of the structure $Syl_2 A_{2^k}$, $Syl_2 A_{n}$ and its property was founded.

\end{document}